\documentclass[11pt,a4paper]{article}

\usepackage[latin1]{inputenc}
\usepackage{latexsym}
\usepackage{amsmath,amssymb,amsfonts,amscd,amsthm}
\usepackage{MnSymbol,stmaryrd}
\usepackage{enumerate}
\usepackage{hyperref}

\newcommand{\MM}{\mathcal{M}}

\newcommand{\smallcup}{\operatorname{\smallsmile}}

\newcommand{\Ad}{\operatorname{Ad}}

\newcommand{\tr}{\operatorname{tr}}
\newcommand{\Ker}{\operatorname{Ker}}

\newcommand{\bock}{\operatorname{\mathbf{b}}}

\renewcommand{\hom}{\operatorname{Hom}}

\newcommand{\SL}{\mathrm{SL}_2}

\def\co{\colon}
{\bigl(
\begin{smallmatrix} }%
{
\end{smallmatrix} \bigr)}

\theoremstyle{change}

\newtheorem{thm}{Theorem}[section]
\newtheorem{prop}[thm]{Proposition}
\newtheorem{lemma}[thm]{Lemma}

\newtheorem{cor}[thm]{Corollary}

\theoremstyle{definition}
       \newtheorem{definition}[thm]{Definition}
       \newtheorem{remark}[thm]{Remark}
       \newtheorem{example}[thm]{Example}
       \newtheorem{question}[thm]{Question}

\usepackage{color}

\newcommand{\new}[1]{{\color{red} #1}}

\usepackage{hyperref}

\title{Representations of knot groups into $\mathrm{SL}_n(\mathbf{C})$ and twisted Alexander polynomials
}

\author{Michael Heusener and 
Joan Porti\thanks{Both authors partially supported
by Mineco through grant MTM2012-34834}}

\date{}

\begin{document}

\maketitle

\begin{abstract}
 Let $\Gamma$ be the fundamental group of the exterior of a knot in the 
three-sphere.
We study deformations of  representations  
 of $\Gamma$ into $\mathrm{SL}_n(\mathbf{C})$ which are the
 sum of two irreducible representations. For such representations we give a 
necessary  condition, in terms of the twisted Alexander polynomial, for the 
existence of irreducible deformations. We also give a more restrictive 
sufficient condition for the existence of irreducible deformations.
 We also prove a duality theorem for twisted Alexander polynomials and
 we describe the local structure of the representation and character varieties.

\end{abstract}

\section{Introduction}
\label{sec:intro}

Let $K\subset S^3$ be an oriented knot in the three-sphere. Its exterior is the compact
three-manifold $X=S^3\setminus {\mathcal{N}}(K)$. 
Set $\Gamma=\pi_1(X)$ and let $\varphi\co \Gamma \twoheadrightarrow \mathbf Z$
denote
the abelianization morphism, so that $\varphi(\gamma)$ is the linking number in
$S^3$ between any loop realizing $\gamma\in \Gamma$ and $K$.  Let
$$
\alpha\co\Gamma\to \mathrm{SL}_a(\mathbf C) \quad\textrm{ and }\quad 
\beta\co\Gamma\to \mathrm{SL}_b(\mathbf C)
$$
be \emph{irreducible} and \emph{infinitesimally regular} representations.

\begin{definition}
\label{def:irreductible}
 A representation $\alpha\co\Gamma\to \mathrm{SL}_a(\mathbf C) $ is called 
\emph{reducible} when it preserves a proper subspace of $\mathbf C^a$, otherwise
it is called
\emph{irreducible}. The representation $\alpha$ is called \emph{semi-simple} or 
\emph{completely reducible} if $\alpha$ is a direct sum of irreducible
representations. 
\end{definition}

In what follows we  call  a representation $\alpha\co\Gamma\to
\mathrm{SL}_a(\mathbf C) $
\emph{infinitesimally regular } if 
 $H^1(\Gamma ; \mathfrak{sl}_a(\mathbf{C})_{\Ad\alpha}))\cong \mathbf C^{a-1}$.

As we assume that $\alpha$ is irreducible and infinitesimally regular, 
its character is a regular point of the character variety of $\Gamma$ in $\mathrm{SL}_a(\mathbf C) $ 
(Proposition~\ref{prop:infinitesimalregularitycharacter}).
When $b=1$, then $\beta$ is trivial  and hence it is  {infinitesimally regular}.

For a given nonzero complex number $\lambda\in \mathbf C^*$ we  consider the  representation
$ \rho_\lambda= (\lambda^{b\varphi}\otimes\alpha)\oplus (\lambda ^{-a\varphi}\otimes\beta)$, namely for all $\gamma\in\Gamma$
\begin{equation}\label{eq:rho_lambda}
\rho_\lambda(\gamma)=
\begin{pmatrix} 
\lambda^{b\varphi(\gamma)}\alpha(\gamma) & \mathbf 0\\ 
\mathbf 0 & \lambda ^{-a\varphi(\gamma)}\beta(\gamma)
\end{pmatrix}\in \mathrm{SL}_n(\mathbf{C})\, ,
\end{equation}
where $a+b=n$. 
The representation $\rho_\lambda\co\Gamma\to \mathrm{SL}_n(\mathbf{C})$ is reducible and the following question then arises:
\begin{question}
When can  $ \rho_{\lambda}$ be  deformed to irreducible
representations?
\end{question}

We  give  necessary and sufficient conditions in terms of twisted Alexander polynomials. For this purpose we consider the representations
$$
\alpha\otimes\beta^*: \Gamma\to \operatorname{Aut} (M_{a\times b}(\mathbf C) )
$$
defined by 
$
(\alpha\otimes\beta^*) (\gamma) (A)= \alpha(\gamma) A \beta(\gamma^{-1})
$,
for $\gamma\in \Gamma$ and $A\in M_{a\times b}(\mathbf C)$. Similarly, consider  
$$
\beta\otimes\alpha^*: \Gamma\to \operatorname{Aut} (M_{b\times a}(\mathbf C) ).
$$
The corresponding twisted Alexander polynomials of degree $i$ are denoted by
$$
\Delta_i^{+}(t)= \Delta_i^{\alpha\otimes \beta^*}(t)
\quad\textrm{ and }\quad
\Delta_i^{-}(t)= \Delta_i^{\beta\otimes \alpha^*}(t)\,.
$$
Recall that the twisted Alexander polynomial is a generator of the order ideal of the 
twisted Alexander module and hence it is unique up to multiplication with an invertible element 
 of the group ring 
$\mathbf{C}[\mathbf{Z}]\cong\mathbf C [t^{\pm 1}]$, i.e.~$c\, t^k$,  with $c\in\mathbf C^*$ and $k\in\mathbf Z$ (see Definition~\ref{def:Alexander} for more details).
We have $\Delta_i^{\pm}(t)=1$ for $i>2 $ and $\Delta_2^{\pm}(t)\in\{0,1\}$. 
We  prove in Corollary~\ref{cor:reductivity} that $\alpha\otimes \beta^*$ is a \emph{semi-simple}
representation, hence by Theorem~\ref{thm:duality} 
we obtain the duality formula (Corollary~\ref{cor:Alex-sym}):
$$
\Delta_i^+(t) \doteq \Delta_i^-(1/t)\,.
$$
Here $p \doteq q$ means that  $p$ and $q$ are \emph{associated elements} in 
$\mathbf C[\mathbf Z]$, i.e.\ there exists some unit 
$c\, t^k\in \mathbf C[\mathbf Z]\cong\mathbf C[t^{\pm 1}]$, 
with $ c\in\mathbf C^*$ and $k\in\mathbf Z$,  such that $p=c\,t^k\,q$.
This duality formula is a particular case of
 Theorem~\ref{thm:duality},
where we establish a duality formula for twisted Alexander polynomials provided that the twisting representation is semi-simple.
This duality formula can also be deduced from results of Friedl, Kim, and Kitayama \cite{FTT12}.
 
The following theorem gives a necessary condition for the deformability of $\rho_\lambda$ 
to irreducible representations:
\begin{thm}
 \label{thm:nec}
 Assume that $\rho_{\lambda}$ can be deformed to irreducible representations. Then 
 $$
 \Delta_1^+(\lambda^n)= \Delta_1^-(\lambda^{-n})=0.
 $$ 
\end{thm}

 The theorem also applies when $\alpha$ or $\beta$ (or both) is trivial. When both $\alpha$ and $\beta$ are trivial, this is a result obtained in 1967
 independently by Burde \cite{Burde1967} and de Rham \cite{dR1967}.
 The key idea is to look at the dimension of the fibre of the algebraic quotient 
 $R(\Gamma,\mathrm{SL}_n(\mathbf C))\to X(\Gamma,\mathrm{SL}_n(\mathbf C))$.
 When $\rho_\lambda$ can be  deformed to irreducible representations, this dimension jumps among characters of reducible representations, and this 
 translates to the twisted Alexander polynomial, by means of the tangent space and cohomology with twisted coefficients.

The next result is a sufficient condition for the deformability of $\rho_\lambda$ to irreducible representations:
\begin{thm}
 \label{thm:suf}
Assume that  $\Delta_0^+(\lambda^n)\neq0$ and that $\lambda ^n$ is a simple root of 
$\Delta_1^+(t)$. Then 
$\rho_{\lambda}$ can be deformed  to irreducible representations.
\end{thm}
Again this theorem and the next one apply for  $\alpha$ and/or $\beta$ trivial.
Theorems~\ref{thm:suf} and \ref{thm:str} are due to \cite{HPS2001} when both
$\alpha$ and $\beta$ are
trivial, and also related results were obtained in 
\cite{Shors1991},
\cite{FK1991}, 
\cite{HK1997},
\cite{HK1998},
\cite{BenAbdelghani2000},  
\cite{bAL2002}, 
\cite{HP05},
\cite{bAHH2010}, and
\cite{HM14}.

The outline of the proof of Theorem~\ref{thm:suf} is the following: 
the hypothesis  implies that there exists a 
representation $\rho^+\in
R(\Gamma,\mathrm{SL}_n(\mathbf C))$
with the same character as $\rho_\lambda$ but not conjugate to it 
(see Corollary~\ref{cor:Burde-deRham}). An
analysis of the cohomology groups allows us to prove that $\rho^+$ is a smooth
point of  $ R(\Gamma,\mathrm{SL}_n(\mathbf C))$.
Among other tools, this uses the vanishing of obstructions to integrability of
Zariski tangent vectors, due to \cite{Goldman1984}, a smoothness result of the
variety of representations  due to \cite{HM14}, and
the non-vanishing of certain cup product (following the ideas of 
\cite{BenAbdelghani2000}).
Once this smoothness result is established, we realize that the dimension of the
space of reducible representations is less than the dimension of the
component of $R(\Gamma,\mathrm{SL}_n(\mathbf C))$
containing $\rho^+$.

Our next result concerns the local structure of the character variety.
Let $\chi_\lambda\in X(\Gamma,\mathrm{SL}_n(\mathbf C))$ denote the character of 
$\rho_\lambda$.
\begin{thm}
 \label{thm:str} Under the hypothesis of Theorem~\ref{thm:suf}, $\chi_{\lambda}$
belongs to precisely two components $Y$ and $Z$ of
$X(\Gamma,\mathrm{SL}_n(\mathbf C))$, that have dimension
 $n-1$ and meet transversally at $\chi_{\lambda}$ along a subvariety of
dimension {$n-2$}. The component $Y$ contains characters of 
irreducible representations and $Z$
consists only of characters of reducible ones. 
\end{thm}
 
 As in  \cite{HPS2001} and \cite{HP05} for $\mathrm{SL}_2(\mathbf C)$ and $
\mathrm{PSL}_2(\mathbf C)$ respectively, the key idea for Theorem~\ref{thm:str} 
is to study the quadratic cone of the representation $\rho_\lambda$, by
identifying certain obstructions to integrability.
 Here we also use Luna's slice theorem, as in \cite{BenAbdelghani2002}.

 We conclude the paper by an explicit description of the component of the variety of 
 irreducible characters
 of the trefoil knot in $\mathrm{SL}_3(\mathbf C)$ that illustrates our results.
 
 \medskip
 
 The paper is organized as follows. Section~\ref{sec:twistedAlex} is devoted to
twisted Alexander modules, and in particular to the duality theorem, 
Theorem~\ref{thm:duality}.
 In Section~\ref{sec:varieties_of_reps} we review some preliminaries on the
 representation
varieties  and in Section~\ref{sec:twisted-cohomology} some
further preliminaries on twisted cohomology and twisted invariants.
 Then in Section~\ref{section:necessary} we prove Theorem~\ref{thm:nec}. The
proof of the sufficient condition, Theorem~\ref{thm:suf}, splits in
Sections~\ref{sec:inf_def_cup} and \ref{sec:regrho}. 
 Theorem~\ref{thm:str} is proved in Section~\ref{sec:nbhood}. Finally in
Section~\ref{section:example} we compute $X(\Gamma,\operatorname{SL}_3(\mathbf
C))$ for $\Gamma$ the fundamental group
 of the trefoil knot exterior.

\paragraph*{Acknowledgements} We are indebted to
Julien Bichon for helpful discussions and we like to thank Simon Riche for pointing out Nagata's result to us \cite{Nagata1961}. 
We also like to thank Stefan Friedl for providing us with the references \cite{FTT12,HSW10}. 
We  are particular thankful to the anonymous referee for his/her thorough review and for having pointed out 
an inaccuracy in the statement of a preliminary version of Theorem~\ref{thm:str}. His/her remarks led to Lemma~\ref{lemma:Z}.
 The first author was supported by the ANR projects ModGroup and SGT (\emph{Structures Géométriques Triangulées}).
 
 \section{Twisted Alexander modules}
 \label{sec:twistedAlex}

  The aim of this section is to introduce twisted Alexander modules and Alexander
polynomials, together with their
main properties.  
 We also give a new result that we will require later:
 a duality theorem for Alexander polynomials twisted by semi-simple
representations.  It relies on Milnor-Franz duality for Reidemeister torsion,
but it is different, as the torsion  is the ratio of the Alexander polynomials.
For further background about twisted Alexander polynomials see  \cite{KirkLivingston99}.
 
A representation of a group $\Gamma$ in a finite-dimensional 
complex vector space $V$ is a homomorphism $\rho\co\Gamma\to \mathrm{GL}(V)$. 
We say that such a map gives $V$ the structure of a $\Gamma$-module. If there is
no ambiguity about the map $\rho$ we call $V$ itself a representation
 of $\Gamma$ and we will often suppress the symbol $\rho$ and write $\gamma\cdot
v$ or $\gamma\,v$ for $\rho(\gamma)(v)$. {Two representations
$\rho\co\Gamma\to \mathrm{GL}(V)$ and $\varrho\co\Gamma\to \mathrm{GL}(W)$ are
called \emph{equivalent} if there exits an isomorphism $T\co V\to W$ such that 
$\varrho(\gamma)\circ T = T\circ\rho(\gamma)$ for all $\gamma\in\Gamma$ i.e.\ if
the $\Gamma$-modules $V$ and $W$ are isomorphic.}

 Our main reference for group cohomology is \cite{Brown1982}. Since we work with left-modules, 
for defining homology consider the right action of the inverse, as in \cite[(2.1)]{KirkLivingston99}.
As the knot exterior $X$ is an Eilenberg--MacLane space,  (co-)homology groups of 
$\Gamma$ and $X$ are naturally identified. In what follows, we will not distinguish between 
$H_i(\Gamma; V)$ and $H_i(X; V)$.

We  give an interpretation of the low dimensional (co-)homology groups. 
The cohomology group in dimension zero is the module of invariants, i.e.\ 
$$
H^0(\Gamma; V)\cong V^\Gamma=\{v\in V\mid \gamma v = v \text{ for all $\gamma\in\Gamma$}\}.
$$
The homology group in dimension zero is the  co-invariant module:
\[
H_0(\Gamma; V)\cong \mathbf Z\otimes_{\mathbf Z[\Gamma]} V\cong V/I V 
\]
where $I\subset \mathbf Z [\Gamma]$ is the augmentation ideal and
$IV\subset V$ is the subspace generated by 
$\{ \gamma \, v-v \mid v \in V,\  \gamma\in\Gamma \}$.

We will make use of  the interpretation of  $H^1(\Gamma;V)$ by means of crossed morphisms, it is well suited for our purpose. 
 A \emph{crossed morphism}  $d\co\Gamma\to V$ is a map that satisfies
$d(\gamma_1\gamma_2)= d(\gamma_1)+\gamma_1 d(\gamma_2) 
$,
$\forall \gamma_1,\gamma_2\in \Gamma$. A crossed morphism $d$ is called  {\emph{principal}} if there exists $v\in V$ satisfying
$
d(\gamma)=\gamma\, v-v
$, 
$\forall \gamma\in\Gamma$.
Crossed morphisms are precisely the \emph{cocycles} of the standard or bar resolution of the $\Gamma$-module $V$, and the  {principal} ones are the \emph{coboundaries}.
Thus the set of crossed morphisms or cocycles is denoted by $Z^1(\Gamma; V)$ and the set of  {principal} crossed morphisms or coboundaries by $B^1(\Gamma; V)$. In particular
the first  cohomology group is
\begin{equation}
 \label{eqn:cohomology}
H^1(\Gamma;V)\cong Z^1(\Gamma; V)/B^1(\Gamma; V).
 \end{equation}

 Let {$\rho\co\Gamma\to\mathrm{GL}(V)$} be a finite dimensional representation of $\Gamma$.
If $ X_\infty \to X$ denotes the infinite cyclic covering, then
$$
H_i( X_\infty; V)
$$
is a finitely generated $\mathbf{C}[\mathbf Z]$-module,
because $X$ is compact and $V$ is finite dimensional.
Here $\mathbf Z$ is the group of deck transformations of the covering
$  X_\infty \to X$. We will sometimes interpret  the elements of $\mathbf{C}[\mathbf Z]$ as Laurent polynomials, by using the isomorphism  $\mathbf C[\mathbf Z]\xrightarrow\cong \mathbf C[t^{\pm 1}]$
that maps the generator $1$ of $\mathbf Z$ to $t$.

\begin{definition}
\label{def:Alexander}
The homology groups  
$
H_i( X_\infty; V)
$ are 
called the 
 \emph{twisted Alexander modules}, 
viewed as $\mathbf C[\mathbf Z]\cong \mathbf C[t^{\pm 1}]$-modules.  
The corresponding orders are the \emph{twisted Alexander polynomials}
$$
\Delta_i^{\rho}(t)\in\mathbf C [t^{\pm 1}]\ ,
$$
they are unique up to multiplication by a unit
$c\, t^k\in \mathbf C[t^{\pm 1}]$,  $k\in\mathbf Z$, $c\in \mathbf C^*$.
\end{definition}
 Recall that the \emph{order} of a finitely generated
$\mathbf C[t^{\pm 1}]$-module 
$$M=\bigoplus_i \mathbf C[t^{\pm 1}]\Big/ p_i(t) \mathbf C[t^{\pm 1}]$$ 
is
$\prod_i p_i(t) $. In particular the order is nonzero if and only if $M$ is a torsion module.
Notice that this is not the same convention as in \cite{KirkLivingston99}.

Due to the indeterminacy in the definition of twisted Alexander polynomials, we shall write
$$
p(t)\doteq q(t)
$$
to denote that the polynomials $p(t), q(t)\in \mathbf C[\mathbf Z]$  are \emph{associated} i.e.\ they are equal  up to multiplication with an element
$\pm t^k\in \mathbf C[\mathbf Z]$,  $k\in\mathbf Z$., $c\in \mathbf C^*$.

\begin{remark}
It follows from a result of M.~Wada \cite[Theorem~2]{Wada94} 
that the twisted Alexander polynomial of
a link exterior twisted by a representation in 
$\mathrm{SL}_n(\mathbf C)$ is well defined up to powers of $\pm t^k$. 
It is also well known that for $n$ even there is  no sign ambiguity.
We shall not need those facts, as we use essentially  the structure 
of the Alexander module.
\end{remark}

Let 
\begin{equation*}
V[\mathbf{Z}]= V\otimes_{\mathbf C[ \Gamma ]} \mathbf C[\mathbf Z]
\end{equation*}
denote the $\Gamma$-module via the representation $\rho\otimes t^\varphi$. Then we have a natural isomorphism
of $\mathbf{C}[\mathbf Z]$-modules
\begin{equation}
\label{eqn:V[Z]}
 H_i(X; V[\mathbf{Z}]) \cong H_i(X_\infty; V)
\end{equation}
(see \cite[Theorem~2.1]{KirkLivingston99}). 
Notice that equivalent representations give rise to isomorphic $\Gamma$-modules and hence to associated Alexander polynomials.

The \emph{dual representation} $\rho^*\co\Gamma\to\mathrm{GL}(V^*)$ is defined in the usual way by
\[
\rho^*(\gamma) (f)  = f\circ \rho(\gamma)^{-1} \text{ for all $\gamma\in\Gamma$ and $f\in V^*=\hom(V,\mathbf C)$.} 
\]
The following lemma is straightforward.
\begin{lemma}
\label{lem:bilinear_dual}
The representations $\rho$ and $\rho^*$ are equivalent
if and only if there exists a non-degenerate bilinear form
$V\otimes V\to\mathbf C$ which is  $\Gamma$-invariant.
\end{lemma}
\begin{example}
\label{ex:sldual}
For any representation $\rho\co\Gamma\to\operatorname{SL}_2(\mathbf C)$, the module $V=\mathbf C^2$ has a skew-symmetric 
non-degenerate bilinear form defined by the determinant. Namely, the vectors $(x_1,x_2)$ and $(y_1,y_2)\in\mathbf{C}^2$
are mapped to
\[\det
  \begin{pmatrix}
   x_1 & y_1 \\ x_2 & y_2
  \end{pmatrix}
.
\]
In view of Lemma~\ref{lem:bilinear_dual}, $ \rho^*$ and $\rho$ are equivalent and hence 
$\Delta_i^\rho \doteq \Delta_i^{\rho^*}$.
\end{example}

Recall from the introduction (see Definition~\ref{def:irreductible}) that a
representation $\rho\co \Gamma\to \mathrm{GL}(V)$ is called semisimple or 
completely reducible
if $\rho$ is the direct sum of irreducible
representations. 

\begin{remark}
A representation $\rho$ is  completely reducible  if and only if each subspace of $V$ stable under $\rho(\Gamma)$ has a $\rho(\Gamma)$-invariant complement.
\end{remark}

 \begin{thm} \label{thm:duality}
 Let $\rho\co\Gamma\to\mathrm{GL}(V)$ be a completely reducible representation. Then
 \[  \Delta_i^\rho (t^{-1}) \doteq \Delta_i^{\rho^*} (t) \,.\] 
 \end{thm}
 
 Example~\ref{ex:reducible-non-sym} below shows that 
 the hypothesis of complete reducibility
is necessary in Theorem~\ref{thm:duality}. 
This duality formula can also be deduced from results of Friedl, Kim, and Kitayama \cite{FTT12}.

The first step in the proof of Theorem~\ref{thm:duality} is the following:

\begin{lemma}
 \label{prop:dualityDelta0} 
 Let $\rho\co\Gamma\to\mathrm{GL}(V)$ be a completely reducible representation.
 The modules $H_0(X_\infty ; V)$ and $H_0(X_\infty ; V^*)$ are finitely
generated torsion modules. 
 In addition,
\[ 
\Delta_0^\rho(t^{-1})\doteq\Delta_0^{\rho^*}(t). 
\]
\end{lemma}
\begin{proof} 
First notice that if $\rho$ is irreducible or completely reducible then the dual representation $\rho^*$ is also irreducible respectively completely reducible since
each proper invariant subspace of $\rho$ corresponds to 
a proper invariant subspace of $\rho^*$ by the \emph{orthogonality} relation.

We have that $H_0( X_\infty, V)\cong V/ \tilde I V $, 
where $\tilde I\subset \mathbf C [\pi_1( X_\infty)]$ is the augmentation ideal.
Hence, $H_0(X_\infty ; V)$ is a finite dimensional 
$\mathbf C$-vector space and as $\mathbf C[t^{\pm 1}]$-module it cannot have a
free summand. This proves that $H_0( X_\infty, V)$ is a finitely generated torsion module.

In order to prove the symmetry relation it is sufficient to prove it  for irreducible representations
since for $\rho_1\co\Gamma\to\mathrm{GL}(V_1)$ and
$\rho_2\co\Gamma\to\mathrm{GL}(V_2)$ we have
\[
(\rho_1\oplus\rho_2)^* = \rho_1^*\oplus \rho_2^*
\ \text{ and }\  
\Delta^{\rho_1\oplus\rho_2}_i \doteq \Delta^{\rho_1}_i\cdot \Delta^{\rho_2}_i.
\]

First we will prove that for every irreducible representation 
$\rho\co\Gamma\to\mathrm{GL}(V)$ with $\dim V>1$ we have
\begin{equation}\label{eq:irred-Delta0}
 \Delta_0^\rho \doteq 1 \doteq \Delta_0^{\rho^*} \,.
\end{equation}
The irreducibility of $\rho$ and $\dim V>1$ imply that 
$IV\subset V$ is a non-trivial $\Gamma$-invariant subspace, and hence 
$IV= V$. It follows that $H_0(\Gamma ; V)=0$.   
Now, for any complex number $\lambda\in\mathbf{C}^*$ the vector space
$V$ becomes a $\Gamma$-module via $\rho\otimes\lambda^\varphi$, i.e.\
for $\gamma\in \Gamma$ and for $ v\in V$ we have
$\rho(\gamma)\otimes\lambda^{\varphi(\gamma)}(v) = \lambda^{\varphi(\gamma)} \rho(\gamma)v$.
This $\Gamma$-module will be denoted by $V_\lambda$. Notice that $V_\lambda$ is also an irreducible $\Gamma$-module since the map  $v\mapsto \lambda^{\varphi(\gamma)} v$ is a homothety of $V$. Moreover, $V_\lambda$ is a non-trivial $\Gamma$-module  and
hence $H_0(\Gamma ; V_\lambda)=0$ for all $\lambda\in\mathbf{C}^*$.
Next, the short exact sequence of $\Gamma$-modules 
\[
 0\to V[\mathbf Z] \xrightarrow{(t-\lambda)\cdot} V[\mathbf Z] \to V_\lambda \to 0
\]
induces a long exact sequence in homology \cite[III.\S6]{Brown1982}:
\[
    \cdots\to
    H_0(\Gamma;V[\mathbf Z]) \xrightarrow{(t-\lambda)\cdot }
    H_0(\Gamma;V[\mathbf Z]) \rightarrow 
    H_0(  \Gamma;V_{\lambda})  \rightarrow 0\,,
\]
and $H_0(  \Gamma;V_{\lambda})=0$ implies that the multiplication by $(t-\lambda)$
is surjective. Hence for all $\lambda\in \mathbf{C}^*$, the module
$H_0(\Gamma;V[\mathbf Z])$ as no $(t-\lambda)$-torsion.
Therefore, $H_0(\Gamma;V[\mathbf Z]) = 0$ and $\Delta_0^\rho =1$.
Finally, $\rho^*$ is also irreducible and $\dim V^* =\dim V>1$. This implies 
in the same way that $\Delta_0^{\rho^*} =1$

Now suppose that $\dim V =1$, i.e.\ $\rho\co\Gamma\to \mathrm{GL}(V)\cong \mathbf C^*$.
Hence $\rho$ is abelian and completely determined by a non-zero-complex number
$\lambda$, 
meaning that 
$\forall\gamma\in\Gamma$ and $\forall v\in V$ we have 
$\rho(\gamma)(v)= \lambda^{\varphi(\gamma)} v$. 
So we write $\rho=\lambda^\varphi$.
Now 
\[ H_0(\Gamma;V[\mathbf Z])\cong V[\mathbf Z]/ I V[\mathbf Z] \cong V[t^{\pm 1}]/(\lambda\,t-1),\]
since $\lambda^\varphi$ is an abelian representation and factors through $\mathbf Z$.
Therefore $\Delta_0^{\lambda^\varphi}(t) \doteq t-\lambda^{-1}$. 
The dual representation $(\lambda^\varphi)^*$ is  $\lambda^{-\varphi} $, as
$(\lambda^\varphi)^*(\gamma)(f)= f\circ (\lambda^{\varphi( \gamma)})^{-1}  =  \lambda^{-\varphi(\gamma)} f$, where
$\gamma\in\Gamma$ and $f\in V^*$. The same calculation as above shows that
$H_0(\Gamma;V^*[\mathbf Z])\cong V[t^{\pm 1}]/(\lambda^{-1}\,t-1)$ and hence
$\Delta_0^{(\lambda^\varphi)^*}(t) \doteq t-\lambda$. 
We obtain $\Delta_0^{(\lambda^\varphi)*}(t) \doteq \Delta_0^{\lambda^\varphi}(t^{-1})$, which proves the lemma.
\end{proof}

\begin{proof}[Proof of Theorem~\ref{thm:duality}]
The knot exterior $X$ has the homotopy type of a $2$-dimensional complex. Therefore $H_i(X_\infty ; V)=0$ for $i>2$ and
 $H_2(X_\infty ; V)$ is a free $\mathbf C[\mathbf Z]$-module.
This implies that $\Delta^\rho_i\doteq 1$ for $i>2$ and $\Delta^\rho_2\in\{0,1\}$.
According to the value of $\Delta^\rho_2$ there are two cases to study.

Assume first that $\Delta^\rho_2=0$. This is equivalent to $H_2(X_\infty ; V)$ being a nontrivial, free $\mathbf C[\mathbf Z]$-module.
By an Euler characteristic argument, $H_1(X_\infty ; V)$ contains also a non-trivial free factor of the same rank.
In particular $\Delta^\rho_1=0$. Since  $H_i(X_\infty ; V)\cong H_i(X ; V[\mathbf Z]) $, the universal coefficient theorem yields that 
$H_i(X; V[\mathbf Z]\otimes_{\mathbf C[\mathbf Z]}\mathbf C(t))\neq 0$ for $i=1,2$. 
Notice also that the natural pairing $V\times V^*\to \mathbf C$ extends to a non-degenerate
$\mathbf C(t)$-bilinear form
\begin{equation}
 \label{eqn:blinear_extended} 
(V[\mathbf Z]\otimes_{\mathbf C[\mathbf{Z}]} \mathbf C(t))
\times 
(V^*[\mathbf Z]\otimes_{\mathbf C[\mathbf{Z}]} \mathbf C(t))
\to \mathbf C(t) \, .
\end{equation}
Using this bilinear form and Poincar\'e duality,
$H_i(X,\partial X; V^*[\mathbf Z]\otimes_{\mathbf C[\mathbf Z]}\mathbf C(t))\neq 0$ for $i=1,2$.
Since the homology of the 2-torus $\partial X$ with coefficients 
$V^*[\mathbf Z]\otimes_{\mathbf C[\mathbf Z]}\mathbf C(t)$ vanishes \cite[\S 3.3]{KirkLivingston99},
$H_i(X; V^*[\mathbf Z]\otimes_{\mathbf C[\mathbf Z]}\mathbf C(t))\neq 0$ for $i=1,2$. 
Hence $\Delta^{\rho^*}_1=\Delta^{\rho^*}_2=0$.

Next we deal with the case $\Delta^\rho_2\doteq 1$. Since this is equivalent to $H_2(X_\infty ; V)=0$, the homology 
argument in the previous paragraph gives $\Delta^{\rho^*}_2\doteq 1$. For the first Alexander polynomials we shall
use Reidemeister torsion and Franz-Milnor duality. 
By Kitano's theorem \cite{Kitano96}
the  torsion of $X$ with coefficients $ V[\mathbf Z]\otimes_{\mathbf{C}[\mathbf Z]}\mathbf C(t)$ is the ratio of
Alexander polynomials:
 $$
  \operatorname{TOR}(X; V[\mathbf Z]\otimes_{\mathbf{C}[\mathbf Z]}\mathbf C(t))
  \doteq  \frac{\Delta_1^\rho }{\Delta_0^\rho },
 $$
see \cite[Thm~3.4]{KirkLivingston99} for this precise statement (this is  a version of Milnor's theorem \cite{Milnor62}, cf.~\cite{Turaev86}).

Using the bilinear form  \eqref{eqn:blinear_extended},
Franz-Milnor's duality for Reidemeister torsion \cite{Milnor62,Franz37} gives
 \begin{align*}
  \operatorname{TOR}(X; V[\mathbf Z]\otimes_{\mathbf{C}[\mathbf Z]}\mathbf C(t) )(t) &\doteq \operatorname{TOR}(X,\partial X; V^*[\mathbf Z]\otimes_{\mathbf{C}[\mathbf Z]}\mathbf C(t))(\frac 1t)
   \\&\doteq 
\frac{\operatorname{TOR}(X;V^*[\mathbf Z]\otimes_{\mathbf{C}[\mathbf Z]}\mathbf C(t))(\frac 1t)}
  {\operatorname{TOR}(\partial X;V^*[\mathbf Z]\otimes_{\mathbf{C}[\mathbf Z]}\mathbf C(t))(\frac 1t)},
 \end{align*}
  see \cite[\S 5.1]{KirkLivingston99}.
 Since $\partial X\cong S^1\times S^1$, 
 $\operatorname{TOR}(\partial X;V^*[\mathbf Z]\otimes_{\mathbf{C}[\mathbf Z]}\mathbf C(t))\doteq 1$ \cite[\S 3.3]{KirkLivingston99}.
 Then the theorem follows from Lemma~\ref{prop:dualityDelta0}.
\end{proof}

 \begin{remark}
 Note that every representation $\rho\co\Gamma\to \mathrm{O}(n)$ is completely
reducible since for each stable subspace $W$ the orthogonal complement $W^\bot$
is also stable. Moreover, we have $\rho^* = \rho$ and hence 
 $\Delta_i^\rho (t^{-1}) \doteq \Delta_i^{\rho} (t)$ is symmetric (see
\cite[Theorem B]{Kitano96}). It follows also from the proof of 
 Lemma~\ref{prop:dualityDelta0} that  $\Delta_0^{\rho} (t) = (t-1)^{k_+}
(t+1)^{k_-}$ where
 $k_{+}=\dim \{ v\in \mathbf R^n \mid \rho(\gamma) v =  v, \   \forall
\gamma\in\Gamma \}$ and
  $k_{-}=\dim \{ v\in \mathbf R^n \mid \rho(\gamma) v = 
(-1)^{\varphi(\gamma)}v, \   \forall \gamma\in\Gamma \}$.  

It was proved in Hillman, Silver, and Williams \cite{HSW10} that 
 $\Delta_i^\rho (t^{-1}) \doteq \Delta_i^{\rho} (t)$ holds if $\rho^*$ and  $\rho$ are conjugates.
 \end{remark}

We finish this section with an example to show that the hypothesis of complete reducibility is needed
  in Theorem~\ref{thm:duality}:
  \begin{example}
 \label{ex:reducible-non-sym}
 We exhibit representations that are not completely reducible and such that 
the conclusion of Theorem~\ref{thm:duality}
fails.
  In order to construct such a representation, we take
 $\rho\co \Gamma\to \operatorname{SL}_2(\mathbf C)$ of the form
   $$
 \rho =  \begin{pmatrix}
          1 & d \\
          0 & 1
         \end{pmatrix}
 \begin{pmatrix} \lambda^\varphi & 0 \\ 0 & \lambda^{-\varphi}\end{pmatrix}
 $$
 that is not abelian. It is a representation if $d\in Z^1(\Gamma,\mathbf C_{\lambda^2})$ 
  and
it is non abelian if $\lambda\neq \pm 1$ and $d\nin B^1(\Gamma,\mathbf C_{\lambda^2})$,
where  $\mathbf C_{\lambda^2}$ denotes the $\Gamma$-module  given by
$\gamma\cdot z = \lambda^{2\varphi(\gamma)} z$ for $\gamma\in\Gamma$, $z\in \mathbf C$, cf.\
Lemma~\ref{lem:repcocycle}.
  Such a representation exists if and only if 
  $\lambda^2$ is a root of the untwisted Alexander polynomial (in particular
$\lambda\neq \pm 1$), see   \cite{Burde1967, dR1967,HPS2001} for instance, or
Lemma~\ref{lem:repcocycle}.
  As $\rho$ is not abelian,  its restriction to $\pi_1(X_\infty)$  is nontrivial
 but 
 $$
 \rho(\pi_1(X_\infty))\subset \left\{\begin{pmatrix} 1 & c \\ 0 & 1 \end{pmatrix} \ \Big|\  c\in \mathbf C\right\}.
 $$
The cohomology module $H_0( X_\infty; \mathbf C^2)$ is isomorphic to ${\mathbf C}^2/I {\mathbf C}^2$. Here 
the subspace $I  {\mathbf C}^2 \subset{\mathbf C}^2$ is generated by  elements 
of the form $v-\rho(\gamma)v$, with $\gamma\in\pi_1(X_\infty)$ and $v\in {\mathbf C}^2$,
i.e.~$I{\mathbf C}^2= \big\{ {c\choose 0} \mid c\in\mathbf C\big\}$.
So, the linear projection $\mathbf C^2\to \mathbf C$ onto the second coordinate induces a linear isomorphism 
${\mathbf C}^2/I{\mathbf C}^2\xrightarrow{\cong}\mathbf C$. The action of a meridian $m\in\Gamma$
on ${\mathbf C}^2/I{\mathbf C}^2$ is multiplication by
$\lambda^{-1}$ and hence
 $
 H_0( X_\infty;{\mathbf C}^2 )\cong \mathbf{C}[t^{\pm1}]/(t-\lambda^{-1})
 $
 as $\mathbf C[\mathbf Z]$-modules.
 Therefore, 
 $\Delta_0^{\rho}(t)=t-\lambda^{-1}$.
 On the other hand, using that every representation in
$\operatorname{SL}_2(\mathbf C)$ is equivalent to its dual, see Example~\ref{ex:sldual},
 $\Delta_0^{\rho^*}(t)\doteq t-\lambda^{-1}$, and 
 $$
 \Delta_0^\rho (t^{-1}) \doteq (t-\lambda) \text{ and }
 \Delta_0^{\rho^*} (t) \text{ are not associated.}
$$

Notice that  if
$\Delta_2^\rho\doteq 1$ then  Franz-Milnor's duality (used in the proof of
Theorem~\ref{thm:duality})
applies and it holds that $\Delta_1^\rho(t^{-1})/ \Delta_0^\rho(t^{-1}) \doteq
\Delta_1^{\rho^*}(t)/\Delta_0^{\rho^*}(t)$. In particular
$\Delta_1^\rho(t^{-1})$ and  $\Delta_1^{\rho^*}(t)$ are not associated either.
\end{example}

 \section{Varieties of representations}
 \label{sec:varieties_of_reps}

 In this section we recall some preliminaries on the varieties of
representations, we discuss representations of 
  the peripheral subgroup $\pi_1(\partial X)\cong\mathbf Z\oplus\mathbf Z$,
 and we state a regularity result,
  Proposition~\ref{prop:MichaelOuardia} due to \cite{HM14}.
 We also show that infinitesimal regularity implies regularity of the
representation (Corollary~\ref{cor:infinitesimalregularity}) and its character
(Proposition~\ref{prop:infinitesimalregularitycharacter}).

Recall that the set of all representations of $\Gamma$ in $ \mathrm{SL}_n(
\mathbf C)$ is called  the \emph{variety of representations} or the
$\mathrm{SL}_n(\mathbf C)$-\emph{representation variety}:
 $$
 R( \Gamma, \mathrm{SL}_n( \mathbf C))=\hom(\Gamma,\mathrm{SL}_n( \mathbf C)).  
 $$
 It is an affine algebraic set (possibly with several components), as $\Gamma$
is finitely generated.
 More precisely, $ R( \Gamma, \mathrm{SL}_n( \mathbf C))$ embeds in a Cartesian
product
 $ \mathrm{SL}_n( \mathbf C)\times\cdots\times  \mathrm{SL}_n( \mathbf C)$ by
mapping each representation to the image of a generating set, and
$\mathrm{SL}_n( \mathbf C)$
 is an algebraic group in $\mathbf C^{n^2}$. The group relations of a
presentation of $\Gamma$ induce the algebraic equations defining $ R( \Gamma,
\mathrm{SL}_n( \mathbf C))$. 
 Different presentations give isomorphic algebraic sets (see \cite{LM85}, for
instance). 
 
 The group $\mathrm{SL}_n( \mathbf C)$ acts on $R( \Gamma, \mathrm{SL}_n(
\mathbf C))$ by conjugation. The 
  algebraic quotient by this action is the \emph{variety of characters} or
  $\mathrm{SL}_n(\mathbf C)$-\emph{character variety}
 $$
 X( \Gamma, \mathrm{SL}_n( \mathbf C))= 
 R( \Gamma, \mathrm{SL}_n( \mathbf C)) \sslash \mathrm{SL}_n( \mathbf C).
 $$
 Recall that the GIT quotient exists since  $ \mathrm{SL}_n( \mathbf C)$ is
\emph{reductive} and the representation variety is an affine algebraic set. (For
more details see \cite[3.\S3]{Newstead1978} or
 \cite{AlgebraicGeometryIV}.)
 
 To describe the Zariski tangent space to $R( \Gamma, \mathrm{SL}_n( \mathbf
C))$ and $X( \Gamma, \mathrm{SL}_n( \mathbf C))$ we use crossed morphisms or
cocycles.

An \emph{infinitesimal deformation} of a representation is the same as a Zariski
tangent vector to $ R( \Gamma, \mathrm{SL}_n( \mathbf C))$. We  use 
Andr\'e Weil's construction, that identifies 
$Z^1(\Gamma;\mathfrak{sl}_n(\mathbf C) _{\Ad\rho} )$ 
with the  Zariski tangent space
to the scheme $\mathcal{R}( \Gamma, \mathrm{SL}_n( \mathbf C))$ at $\rho$.
Here $\mathfrak{sl}_n(\mathbf C) _{\Ad\rho}$ is a $\Gamma$-module via the adjoint action i.e.
$\gamma\cdot x = \Ad_{\rho(\gamma)}(x)$ for $\gamma\in\Gamma$ and 
$x\in\mathfrak{sl}_n(\mathbf C)$. 
Notice furthermore that the algebraic equations defining the representation
variety may be non-reduced, hence there is an underlying affine scheme 
$\mathcal{R}( \Gamma, \mathrm{SL}_n( \mathbf C))$ with a possible non-reduced
coordinate ring.
Weil's construction assigns to each cocycle 
$d\in Z^1(\Gamma;\mathfrak{sl}_n(\mathbf C))$ the infinitesimal deformation
 $
\gamma\mapsto  (1+ \varepsilon\, d(\gamma)) \rho(\gamma), \
\forall\gamma\in\Gamma,
 $
 which satisfies the defining equations for $R( \Gamma, \mathrm{SL}_n( \mathbf
C))$ up to terms in the ideal
$(\varepsilon^2)$ of $\mathbf C[\varepsilon]$, i.e. a Zariski tangent vector
to 
$\mathcal{R}( \Gamma, \mathrm{SL}_n( \mathbf C))$.
Weil's construction identifies $B^1(\Gamma;\mathfrak{sl}_n(\mathbf C)_{\Ad\rho}
)$ with the tangent space to the orbit by conjugation. See \cite{Weil64},  \cite{LM85}, and
\cite{BenAbdelghani2002} for more details.

Let $\dim_\rho R(\Gamma, \mathrm{SL}_n( \mathbf C))$ denote the local dimension
of 
$R(\Gamma, \mathrm{SL}_n( \mathbf C))$ at $\rho$ (i.e.\ the maximal dimension of
the irreducible components of $R(\Gamma, \mathrm{SL}_n( \mathbf C))$ containing
$\rho$ 
\cite[Ch.~II]{Shafarevich1977}). So we obtain:
\begin{equation}\label{eq:regular-point} 
\dim_\rho R(\Gamma, \mathrm{SL}_n( \mathbf C)) \leq 
\dim T_\rho(R(\Gamma, \mathrm{SL}_n( \mathbf C)))\leq 
\dim Z^1(\Gamma ; \mathfrak{sl}_n( \mathbf C)_{\Ad\rho})\,.
\end{equation}

\begin{definition}\label{def:regular}
Let $\rho\co\Gamma\to\mathrm{SL}_n(\mathbf C)$ be a representation.
We say that  $\rho$ is a \emph{regular point} of
the representation variety if 
\[ \dim_\rho R(\Gamma, \mathrm{SL}_n( \mathbf C)) =
\dim Z^1(\Gamma ; \mathfrak{sl}_n( \mathbf C)_{\Ad\rho})\,.\] 
We call $\rho$ \emph{infinitesimal regular} if 
$\dim H^1(\Gamma ; \mathfrak{sl}_n( \mathbf C)_{\Ad\rho}) = n-1$.
\end{definition}
 
It follows directly
from \eqref{eq:regular-point} that a regular point is a smooth point of the
representation variety. 
There  are representations of discrete groups which are smooth
points of the representation variety without being regular, as the scheme
$\mathcal{R}( \Gamma, \mathrm{SL}_n( \mathbf C))$ may be non reduced.
(See
\cite[Example~2.10]{LM85} for more details.)

We also make use of the Poincar\'e-Lefschetz duality theorem with twisted coefficients:
let $M$ be a connected, orientable, compact $m$-dimensional manifold with boundary $\partial M$ and let
$\rho\co\pi_{1}(M)\to \mathrm{SL}_n(\mathbf C)$ be a representation. Then the cup-product and the Killing form 
$b\co\mathfrak{sl}_n(\mathbf C)\otimes\mathfrak{sl}_n(\mathbf C)\to\mathbf C$ 
 induce a non-degenerate bilinear pairing
\begin{multline}\label{eq:poincare} H^k(M ;  \mathfrak{sl}_n(\mathbf C)_{\Ad\rho})\otimes H^{m-k}(M,\partial M ; \mathfrak{sl}_n(\mathbf C)_{\Ad\rho})\xrightarrow{\smallsmile}\\
H^{m}(M,\partial M ; \mathfrak{sl}_n(\mathbf C)_{\Ad\rho}\otimes\mathfrak{sl}_n(\mathbf C)_{\Ad\rho})\xrightarrow{b}\\
H^{m}(M,\partial M ; \mathbf C)\cong \mathbf C
\end{multline}
and hence an isomorphism 
$$H^k(M ; \mathfrak{sl}_n(\mathbf C)_{\Ad\rho})\cong H^{m-k}(M,\partial M ; \mathfrak{sl}_n(\mathbf C)_{\Ad\rho})^*,$$ 
for all $0\leq k\leq m$.
See \cite{Johnson-Millson1987,Porti} for more details.

\begin{lemma}\label{lem:H1-Z+Z}
For any representation $\varrho\co\mathbf Z\oplus\mathbf Z\to\mathrm{SL}_n(\mathbf{C})$ we have:
\[ \dim H^1(\mathbf Z\oplus\mathbf Z ;\mathfrak{sl}_n(\mathbf C)_{\Ad \varrho}) \geq 2(n-1)\,.\]
In addition, $\dim H^1(\mathbf Z\oplus\mathbf Z,\mathfrak{sl}_n(\mathbf C))= 2(n -1)$
if and only if $\varrho$ is a regular point of 
 $R(\mathbf Z\oplus\mathbf Z,  \mathrm{SL}_n(\mathbf C))$.
\end{lemma}
Recall that a function $\phi\co R(\Gamma,\mathrm{SL}_n(\mathbf{C}))\to \mathbf Z$ is 
called \emph{upper semi-continuous} if for all $k\in\mathbf Z$ the set 
 $\phi^{-1}\big([k,\infty)\big)$ is closed.
Moreover, it is easy to prove that for $q=0,1$ the function
$\rho \mapsto \dim H^q(\Gamma;\mathfrak{sl}_n(\mathbf C)_{\Ad \rho})$ 
is upper semi-continuous  (see \cite[Lemma~3.2]{Heusener-Porti2011}, 
this is a particular case of the semi-continuity theorem \cite[Ch.~III, Theorem 12.8]{Hartshorne1977}).
\begin{proof}[Proof of Lemma~\ref{lem:H1-Z+Z}]
Poincar\'e duality and Euler characteristic give
\[
\frac 1 2 \dim H^1(\mathbf Z\oplus\mathbf Z;\mathfrak{sl}_n(\mathbf C)_{\Ad \varrho}) =  
\dim H^0(\mathbf Z\oplus\mathbf Z;\mathfrak{sl}_n(\mathbf C)_{\Ad \varrho})
=  \dim \mathfrak{sl}_n(\mathbf C)^{\mathbf Z\oplus\mathbf Z}\,.
\]
By a result of Richardson \cite[Thm.~C]{Richardson1979},  every representation of 
$\mathbf Z\oplus\mathbf Z$ into $\mathrm{SL}_n(\mathbf C)$ is a limit of diagonal representations,
and  for diagonal representations $\dim \mathfrak{sl}_n(\mathbf C)^{\mathbf Z\oplus\mathbf Z}\geq n-1$.
The upper semi-continuity  of the function
$\varrho \mapsto \dim H^0(\mathbf Z\oplus\mathbf Z;\mathfrak{sl}_n(\mathbf C)_{\Ad \varrho})$ gives the inequality in general.

For the second statement, in the same paper 
Richardson \cite[Thm.~C]{Richardson1979} also proved that the 
representation variety $R(\mathbf Z\oplus\mathbf Z,  \mathrm{SL}_n(\mathbf C))$ is an irreducible algebraic variety of dimension 
$(n+2)(n-1)$. It follows that $\varrho\in R(\mathbf Z\oplus\mathbf Z , \mathrm{SL}_n(\mathbf C))$ is a regular point iff
$\dim Z^1(\mathbf Z\oplus\mathbf Z,\mathfrak{sl}_n(\mathbf C))= (n+2)(n-1)$.

On the other hand, 
\begin{align*}
 \dim Z^1(\mathbf Z\oplus\mathbf Z,\mathfrak{sl}_n(\mathbf C)) &=
\dim H^1(\mathbf Z\oplus\mathbf Z,\mathfrak{sl}_n(\mathbf C)) + \dim B^1(\mathbf Z\oplus\mathbf Z,\mathfrak{sl}_n(\mathbf C)) ;\\
 \dim B^1(\mathbf Z\oplus\mathbf Z,\mathfrak{sl}_n(\mathbf C)) &= n^2-1 - 
 \dim H^0(\mathbf Z\oplus\mathbf Z,\mathfrak{sl}_n(\mathbf C)) ;\\ 
\dim H^0(\mathbf Z\oplus\mathbf Z,\mathfrak{sl}_n(\mathbf C)) &= 
\frac 1 2 \dim H^1(\mathbf Z\oplus\mathbf Z,\mathfrak{sl}_n(\mathbf C))\,.
\end{align*}
Hence 
\[ \dim Z^1(\mathbf Z\oplus\mathbf Z,\mathfrak{sl}_n(\mathbf C))=
\frac 1 2 \dim H^1(\mathbf Z\oplus\mathbf Z,\mathfrak{sl}_n(\mathbf C)) + n^2 -1\,.\]
Thus the lemma follows. (See also \cite{Popov2008}.)
 \end{proof}

%
We will require the following proposition \cite[Proposition 3.3]{HM14}:
\begin{prop}[{\cite{HM14}}]
 \label{prop:MichaelOuardia}
 Let $\alpha\in R(\Gamma, \mathrm{SL}_a(\mathbf C))$ be an infinitesimally regular representation.
 Then $\alpha$ is a regular point of the $\mathrm{SL}_a(\mathbf C)$-representation variety $ R(\Gamma, \mathrm{SL}_a(\mathbf C))$ and it belongs to a unique component of dimension 
 $a^2+a-2-\dim H^0(\Gamma;\mathfrak{sl}_a(\mathbf C) )$.
 \end{prop}
\begin{remark}
 \label{rem:irreducibleH0}
For an irreducible representation  $\alpha\co\Gamma\to \mathrm{SL}_a(\mathbf C)$,  it holds that
$
H^0(\Gamma ; \mathfrak{sl}_a(\mathbf{C})_{\Ad\alpha}))=0
$. 
 Indeed, if $X\in\mathfrak{sl}_a(\mathbf{C})$  commutes with  $\alpha(\gamma)$ for all 
 $\gamma\in\Gamma$, then Schur's lemma implies that $X$ is a scalar matrix and hence $X=0$.
 \end{remark}
As a corollary we obtain from Proposition~\ref{prop:MichaelOuardia} and Remark~\ref{rem:irreducibleH0}: 
\begin{cor}
\label{cor:infinitesimalregularity}
If an \emph{irreducible} representation  $\alpha\co\Gamma\to \mathrm{SL}_a(\mathbf C)$ is infinitesimally regular then it is a regular point of $R(\Gamma,\mathrm{SL}_a(\mathbf C))$
of local dimension $a^2+a-2$.
\end{cor}

One has furthermore:

\begin{prop}
\label{prop:infinitesimalregularitycharacter}
If an irreducible representation  $\alpha\co\Gamma\to \mathrm{SL}_a(\mathbf C)$ is infinitesimally regular, then its character is a  {smooth} point of 
$X(\Gamma,\mathrm{SL}_a(\mathbf C))$ of local dimension $a-1$.
\end{prop}

\begin{proof}
  By Corollary~\ref{cor:infinitesimalregularity}, $\alpha$ is a regular point of  $R(\Gamma,\mathrm{SL}_a(\mathbf C))$ of local dimension $a^2+a-2$. 
As $\alpha$ is irreducible, the fibre of the projection $R(\Gamma, \mathrm{SL}_a(\mathbf C))\to X(\Gamma,  \mathrm{SL}_a(\mathbf C))$ at $\alpha$ has dimension
  $a^2-1$. The dimension of this fibre is an upper semi-continuous function, therefore 
  the dimension of $X(\Gamma,\mathrm{SL}_a(\mathbf C))$ at $\alpha$ is at least $a-1$. On the other hand,
  the dimension of the Zariski tangent space of  $X(\Gamma,\mathrm{SL}_a(\mathbf C))$ at $\alpha$ is at most $\dim H^1(\Gamma; \mathfrak{sl}_a(\mathbf C)_{\Ad\alpha} )$
  (this follows from Luna's slice as $\alpha$ is irreducible, cf.~\cite[Theorem~2.15]{LM85}). Hence we have equality of dimensions and the proposition follows. 
\end{proof}

 \section{Twisted cohomology and twisted polynomials}
 \label{sec:twisted-cohomology}
 
 In this section we prove that $\alpha\otimes\beta^*$ and $\beta\otimes\alpha^*$
are completely reducible representations, so that the duality theorem (Thm.~\ref{thm:duality}) applies to them. {Our assumption that $\alpha\co\Gamma\to\mathrm{SL}_a(\mathbf C)$ and
$\beta\co\Gamma\to\mathrm{SL}_b(\mathbf C)$ are irreducible will be crucial for the conclusion.}

 \paragraph{Decomposition of $\mathfrak{sl}_n(\mathbf C)$.}
Consider the action of $\Gamma$ on the space of matrices with $a$ rows and $b$
columns $M_{a\times b}(\mathbf C)$:
\begin{equation}\label{eq:M^+}
\begin{array}{ccl}
 \Gamma\times M_{a\times b}(\mathbf C) & \to & M_{a\times b}(\mathbf C) \\[1ex]
     (\gamma, A) & \mapsto & \lambda^{n\varphi(\gamma)}\alpha(\gamma) A
\beta(\gamma^{-1})\,.
\end{array}
\end{equation}
The corresponding $\Gamma$-module is denoted by:
$$
\MM^+_{\lambda^n}= M_{a\times b}(\mathbf
C)_{\alpha\otimes\beta^*\otimes\lambda^{n\varphi}}.
$$
Similarly, we consider the module
$$
\MM^-_{\lambda^{-n}}= M_{b\times a}(\mathbf
C)_{\beta\otimes\alpha^*\otimes\lambda^{-n\varphi}} .
$$ 
Notice that those modules occur as factors in the decomposition of
$\mathfrak{sl}_n(\mathbf C)$ as $\Gamma$-modules via the adjoint action
$\Ad\rho_\lambda$:
$$
\mathfrak{sl}_n(\mathbf C)_{\Ad\rho_{\lambda}}= \mathfrak{sl}_a(\mathbf
C)_{\Ad\alpha}\oplus \mathfrak{sl}_b(\mathbf C)_{\Ad\beta}\oplus \mathbf C 
\oplus  \MM^+_{\lambda^{n}}\oplus  
\MM^-_{\lambda^{-n}}.
$$
This can be visualized as
$$
\mathfrak{sl}_n(\mathbf C)_{\Ad\rho_\lambda} \cong
\begin{pmatrix}
  \mathfrak{sl}_a(\mathbf C)_{\Ad\alpha} & \MM^+_{\lambda^{n}} \\
   \MM^-_{\lambda^{-n}} & \mathfrak{sl}_b(\mathbf C)_{\Ad\beta}
\end{pmatrix} \oplus 
\mathbf C 
\begin{pmatrix}
    b \operatorname{Id}_a & 0 \\
    0 & -a \operatorname{Id}_b
\end{pmatrix}\,.
$$

\paragraph{Duality.} For every $\lambda\in\mathbf C^*$ we have a non-degenerate bilinear form:
\begin{equation}
\label{eqn:bilinearform}
\begin{array}{rcl}
  \Psi: \MM^+_{\lambda^{n}}\times \MM^-_{\lambda^{-n}} & \to & \mathbf C \\
   (A,B) &  \mapsto & \tr (AB)
   \end{array}
\end{equation}
which is $\Gamma$-invariant:
$
\Psi(A,B)=\Psi(\gamma A,\gamma B)
$, $\forall\gamma\in \Gamma$.
As an immediate consequence, we have Poincar\' e and Kronecker dualities:
\begin{eqnarray}
H_i(X;\MM^{\pm}_{\lambda^{\pm n}}) & \cong  & 
H_{3-i}(X,\partial X;\MM^{\mp}_{\lambda^{\mp n}})^*; \\
H^i(X;\MM^{\pm}_{\lambda^{\pm n}}) & \cong  & H^{3-i}(X,\partial X;\MM^{\mp}_{\lambda^{\mp n}})^*; \\
H_i(X;\MM^{\pm}_{\lambda^{\pm n}}) & \cong  & H^{i}(X;\MM^{\mp}_{\lambda^{\mp n}})^*. \label{eqn:Kronecker}
 \end{eqnarray}

  The $i$-th twisted Alexander polynomials of the $\Gamma$-modules $\MM^{\mp}_1$ 
  are denoted $\Delta^{\pm}_{i}$. Namely
  \begin{equation*}
   \Delta^+_i  =  \Delta^{\alpha\otimes\beta^*}_i \quad\text{ and }\quad
   \Delta^-_i  =  \Delta^{\beta\otimes\alpha^*}_i \, .
  \end{equation*}
 Taking $\rho=\alpha\otimes\beta^*$, then $\rho^*=\beta\otimes\alpha^*$ by \eqref{eqn:bilinearform}.
In order to apply Theorem~\ref{thm:duality} to those polynomials, we need to show 
that $\rho=\alpha\otimes\beta^*$ is completely reducible; this motivates 
the next paragraph.

\paragraph{Linear algebraic groups.} 
We follow  Humphreys' book \cite{Humphreys1975} as   general reference for linear algebraic groups.
A linear algebraic group $G$ contains a unique largest normal solvable subgroup, which is automatically closed.
Its identity component is then the largest connected normal solvable subgroup of $G$; it is called the \emph{radical} of $G$, 
denoted by $R(G)$. The subgroup of unipotent elements in $R(G)$ is normal in both $R(G)$ and $G$; 
it is called the \emph{unipotent radical} of $G$, denoted by $R_u(G)$. We have that 
$R(G)/R_u(G)$ is a torus. Hence $R(G)$ is a torus if and only if $R_u(G)$ is trivial.

Recall that a representation  $\rho\co\Gamma\to\mathrm{SL}(V)$ is called completely reducible if it is a direct sum of 
irreducible representations, see Definition~\ref{def:irreductible}.
In what follows we will make use of the following result of M.~Nagata:
\begin{thm}[Nagata {\cite[Thm.~3]{Nagata1961}}]\label{thm:Nagata}
Let $G\subset \mathrm{GL}_n(\mathbf C)$ be an algebraic group. 
Then $R_u(G)$ is trivial
if and only if each rational representation of $G$ is completely reducible.\hfill$\Box$
\end{thm} 
Here a representation $\rho\co G\to\mathrm{GL}(V)$  is called \emph{rational} if with respect to a basis of $V$ the matrix entries of $\rho(g)$ are 
{polynomial functions in the $n^2+1$ coordinate 
functions $x_{ij}$, $1\leq i,j\leq n$ and $1/\det$ of $\mathrm{GL}_n(\mathbf C)$.}

\begin{remark}
A non-trivial, connected algebraic group $G$ is called \emph{reductive} if $R_u(G)$ is trivial.
Since the Zariski closure of a matrix group  is in general not connected we will avoid the term reductive in what follows.  
\end{remark}

\begin{lemma}
\label{lem:reductive1}
Let $\Gamma$ be a group and let $\rho\co\Gamma\to \mathrm{SL}_n(\mathbf C)$ be an irreducible representation.
Then the unipotent radical $R_u(G)$ of the Zariski closure 
$G$ of $\rho(\Gamma)\subset \mathrm{SL}_n(\mathbf C)$ is trivial. 
\end{lemma}
\begin{proof}
Suppose that  $R_u(G)\subset \mathrm{SL}_n(\mathbf C)$ is non-trivial.
Every unipotent subgroup of $\mathrm{GL}_n(\mathbf C)$ has a nonzero vector fixed by all elements of the group (see \cite[17.5]{Humphreys1975}).
 Then  the subspace $W\subset \mathbf C^n$ of fixed vectors of $R_u(G)$ is nonzero. 
 By normality, this subspace is preserved by $G$, hence by $\rho(\Gamma)$, which contradicts the irreducibility of $\rho$.
 \end{proof}

\begin{lemma}
\label{lem:reductive2}
 Let
$\alpha\co\Gamma\to \mathrm{SL}_a(\mathbf C)$ and 
$\beta\co\Gamma\to \mathrm{SL}_b(\mathbf C)$
be irreducible. 
Then the unipotent radical $R_u(G)$ of the Zariski closure
$G$ of $(\alpha\oplus\beta)(\Gamma)\subset \mathrm{SL}_a(\mathbf C)\times \mathrm{SL}_b(\mathbf C)$ is trivial.
\end{lemma}

\begin{proof}
    Let $p_a\co \mathrm{SL}_a(\mathbf C)\times \mathrm{SL}_b(\mathbf C)\to \mathrm{SL}_a(\mathbf C)$ denote the projection.
        Then $p_a((\alpha\oplus\beta)(\Gamma)) = \alpha(\Gamma)$ and therefore
 $p_a(R_u(G))$ is contained in the unipotent radical $R_u(G_a)$ of the Zariski closure $G_a$ of 
 $\alpha(\Gamma)$ in $\mathrm{SL}_a(\mathbf C)$. 
(The image of an unipotent element under a morphism of algebraic groups is unipotent 
\cite[15.3]{Humphreys1975}.) 
Now, $R_u(G_a)$ is trivial by Lemma~\ref{lem:reductive1} and hence  $p_a(R_u(G))$ is trivial.
It follows in the same way that $p_b(R_u(G))$ is trivial and hence $R_u(G)=\{1\}$.
 \end{proof}

 \begin{remark}
   The same argument of Lemma~\ref{lem:reductive2}  proves that the Zariski closure of a completely reducible linear representation has trivial unipotent radical.
 \end{remark}

 \begin{cor}
  \label{cor:reductivity} The $\Gamma$-modules $\MM^\pm_{\lambda^{\pm n}}$ are completely reducible. 
 \end{cor}
 \begin{proof}
 By Lemma~\ref{lem:reductive2} the unipotent radical $R_u(G)$ of the Zariski closure
$G$ of $(\alpha\oplus\beta)(\Gamma)\subset \mathrm{SL}_a(\mathbf C)\times \mathrm{SL}_b(\mathbf C)$ is trivial. Hence 
Nagata's theorem (Thm.~\ref{thm:Nagata}) implies that every rational representation of $G$ is completely reducible. In 
particular, the restriction to $G$  of the rational representation
$\mathrm{SL}_a(\mathbf C)\times \mathrm{SL}_b(\mathbf C)\to 
\mathrm{GL}(M_{a\times b}(\mathbf C))$, given by
\[ (A,B)\cdot X = A\,X\,B^{-1}, 
\]
$\forall (A , B)\in \mathrm{SL}_a(\mathbf C)\times \mathrm{SL}_b(\mathbf C)$,
$\forall X\in M_{a\times b}(\mathbf C)$,
is completely reducible.

Since $(\alpha\oplus\beta)(\Gamma)$ is Zariski dense in $G$, we obtain
 that $\MM^+_1$ is a completely reducible $\Gamma$-module.
Finally,  the
action of $\gamma\in\Gamma$ on  $X\in\MM^+_{\lambda^{n}}$, 
given by Equation~\eqref{eq:M^+}, and the
action $\gamma\cdot X =  \alpha(\gamma) \,X\,\beta(\gamma^{-1})$ 
differ only by a homothety. Therefore, $\MM^+_{\lambda^{n}}$ is a completely reducible 
$\Gamma$-module. The proof for $\MM^-_{\lambda^{-n}}$ is similar.
 \end{proof}

\begin{cor}\label{cor:Alex-sym}
We have
\[ \Delta_i^+(t) \doteq \Delta_i^-(t^{-1})\,.\]
\end{cor}

\begin{proof}
The corollary follows directly from Theorem~\ref{thm:duality} and Corollary~\ref{cor:reductivity}.
\end{proof}
 
 \section{Necessary condition}
\label{section:necessary}
 
 The goal of this section is to prove Theorem~\ref{thm:nec}. 
 More precisely, we will prove that 
if the representation
$ \rho_\lambda= (\lambda^{b\varphi}\otimes\alpha)\oplus (\lambda ^{-a\varphi}\otimes\beta)$, as defined in \eqref{eq:rho_lambda},
 can be deformed to irreducible representations, then 
 $\Delta_1^+(\lambda^n)=0$.
 Recall that throughout the paper we assume that $\alpha$ and $\beta$ are irreducible and infinitesimally regular.

 \begin{lemma}
\label{lem:necessary_dimZ1} Assume that $\rho_{\lambda}$ belongs to a component of 
 $R(\Gamma,\mathrm{SL}_n(\mathbf C))$ that contains irreducible representations. Then
 \[ \dim Z^1(\Gamma ; \mathfrak{sl}_n(\mathbf C)_{\Ad \rho_\lambda}) \geq n^2 + n -2\,.\]
 \end{lemma}
 \begin{proof}
It is sufficient to prove the inequality for an irreducible representation 
$\rho\in R(\Gamma,\mathrm{SL}_n(\mathbf C))$, because the dimension of 
$Z^1(\Gamma ; \mathfrak{sl}_n(\mathbf C)_{\Ad \rho})$ is an upper semi-continuous  
function on  $\rho$
and because irreducibility is a Zariski open condition.
 We have 
 \[
 \dim Z^1(\Gamma ; \mathfrak{sl}_n(\mathbf C)_{\Ad \rho}) =
 \dim H^1(\Gamma ; \mathfrak{sl}_n(\mathbf C)_{\Ad \rho}) +
 \dim B^1(\Gamma ; \mathfrak{sl}_n(\mathbf C)_{\Ad \rho})\,.
 \]
 Now, $ \dim B^1(\Gamma ; \mathfrak{sl}_n(\mathbf C)_{\Ad \rho}) = n^2-1$ because $\rho$ is irreducible.
 
Next we apply Poincar\' e duality  to the long exact sequence of the pair $(X,\partial X)$:
\begin{equation}
 \label{eqn:pair}
H^1(X;\mathfrak{sl}_n(\mathbf C)_{\Ad \rho}) \to H^1(\partial X;\mathfrak{sl}_n(\mathbf C)_{\Ad \rho})
\to H^2( X,\partial X;\mathfrak{sl}_n(\mathbf C)_{\Ad \rho}). 
\end{equation}
{Poincar\'e duality \eqref{eq:poincare} implies that $H^1( X;\mathfrak{sl}_n(\mathbf C)_{\Ad \rho})$ and the dual space 
$H^2(X,\partial X;\mathfrak{sl}_n(\mathbf C)_{\Ad \rho})^*$ are isomorphic. Moreover,
 the maps of \eqref{eqn:pair} are dual to each other.} So:
\begin{equation}\label{eqn:dimH^1partialX}
\frac 1 2 \dim H^1(\partial X;\mathfrak{sl}_n(\mathbf C)_{\Ad \rho}) \leq  \dim H^1(\Gamma;\mathfrak{sl}_n(\mathbf C)_{\Ad \rho}).
\end{equation}
The claimed inequality of the statement follows from Lemma~\ref{lem:H1-Z+Z}.
\end{proof}

\begin{lemma}
\label{lem:dimh1-dimh0}
Under the hypothesis of Lemma~\ref{lem:necessary_dimZ1} we have 
\[ \dim H^1(\Gamma ; \MM^+_{\lambda^{n}}) > \dim H^0(\Gamma ; \MM^+_{\lambda^{n}}), \]
or 
\[\dim H^1(\Gamma ; \MM^-_{\lambda^{- n}}) > \dim H^0(\Gamma ; \MM^-_{\lambda^{-n}})\,. \]
\end{lemma}

We shall see in Remark~\ref{rem:bothinequalities} below that we get both inequalities.

 \begin{proof}
Here we use the decomposition of $\Gamma$-modules (see Section~\ref{sec:twisted-cohomology}):
\begin{equation}
\label{eqn:decomposition}
\mathfrak{sl}_n(\mathbf C)_{Ad\,\rho_\lambda}=
\mathfrak{sl}_a(\mathbf C)_{\mathrm{Ad}\,\alpha} \oplus 
\mathfrak{sl}_b(\mathbf C)_{ \mathrm{Ad}\,\beta} \oplus 
\mathbf{C}\oplus \MM^+_{\lambda^{n}}\oplus\MM^-_{\lambda^{-n}}.
\end{equation}
We aim to apply Lemma~\ref{lem:necessary_dimZ1}, so we compute the dimension of the space of $1$-cocycles 
for each $\Gamma$-module in \eqref{eqn:decomposition}.
For each $\Gamma$-module $\mathfrak m$, we use the formula
\begin{multline}
\label{eqn:dimZ1}
\dim Z^1(\Gamma ;  \mathfrak{m})=\dim H^1(\Gamma ;  \mathfrak{m})+\dim B^1(\Gamma ;  \mathfrak{m})
\\ = \dim H^1(\Gamma ;  \mathfrak{m})+ \dim \mathfrak{m}- \dim H^0(\Gamma ;  \mathfrak{m}).
\end{multline}
  Ordering the terms as they appear in \eqref{eqn:dimZ1}: 
\begin{eqnarray*}
\dim Z^1(\Gamma ;  \mathfrak{sl}_a(\mathbf C)_{\mathrm{Ad}\,\alpha}) & = & (a-1)+ (a^2-1) - 0, \\
\dim Z^1(\Gamma ;  \mathfrak{sl}_b(\mathbf C)_{\mathrm{Ad}\,\alpha}) & = & (b-1)+ (b^2-1) - 0, \\
\dim Z^1(\Gamma ;  \mathbf C ) & = & 1+ 1 - 1, \\
\dim Z^1(\Gamma ;  \MM^\pm_{\lambda^{\pm n}}) & = &    \dim H^1(\Gamma ;  
\MM^\pm_{\lambda^{\pm n}}) + a\, b  -
\dim H^0(\Gamma ;  \MM^\pm_{\lambda^{\pm n}}).
\end{eqnarray*}
The first two lines use that $\alpha$ and $\beta$ are  irreducible and infinitesimally regular, 
the last one that
$\dim \MM^\pm_{\lambda^{\pm n}}=a\, b$. Adding up the dimensions of the terms in \eqref{eqn:decomposition}
and using Lemma~\ref{lem:necessary_dimZ1} and the fact that $a+b=n$, we obtain 
\begin{multline*}
n^2+n-2 \leq n^2+n-3 + \dim H^1(\Gamma ;  \MM^+_{\lambda^{n}})  -
\dim H^0(\Gamma ;  \MM^+_{\lambda^{n}})+ \\
\dim H^1(\Gamma ;  \MM^-_{\lambda^{-n}}) -
\dim H^0(\Gamma ;  \MM^-_{\lambda^{-n}}),
\end{multline*}
which proves the lemma.
 \end{proof}
 
 For later use we remark the following computation, made during the last proof. Notice that it does not use that $\rho_{\lambda}$
 can be deformed to irreducible representations (but it uses that $\alpha$ and $\beta$ are irreducible and infinitesimally regular):
\begin{multline}\label{eq:dimZ1}
\dim Z^1(\Gamma;\mathfrak{sl}_n(\mathbf C)_{\Ad\rho_\lambda}) =n^2+n-3 + 
\dim H^1(\Gamma ;  \MM^+_{\lambda^n})  -
\dim H^0(\Gamma ;  \MM^+_{\lambda^n})+ \\
\dim H^1(\Gamma ;  \MM^-_{\lambda^{-n}}) -
\dim H^0(\Gamma ;  \MM^-_{\lambda^{-n}}).
\end{multline}

 \begin{lemma}
   \label{lem:delta10} Let $\rho_\lambda\co\Gamma\to\mathrm{SL}_n(\mathbf C)$ be given by
    $ \rho_\lambda= (\lambda^{b\varphi}\otimes\alpha)\oplus (\lambda ^{-a\varphi}\otimes\beta)$. 
Then
  $\dim H^1(\Gamma ; \MM^\pm_{\lambda^{\pm n}})> \dim H^0(\Gamma ; \MM^\pm_{\lambda^{\pm n}})$ if and only if 
  $\Delta_1^{\mp}(\lambda^{\mp n})=0$. 
    \end{lemma}
  \begin{proof}
  Recall that $\Delta_i^\pm$ is the order of $H_i(X_\infty; \MM_1^\pm)
   \cong H_i(X; \MM^{\pm}_1[\mathbf Z]) $.
   We have a short exact sequence of $\Gamma$-modules
   \[ 0 \to \MM_1^{-}[\mathbf Z] \xrightarrow{(t-\lambda^{-n})\cdot} \MM_1^{-}[\mathbf Z] \to \MM^-_{\lambda^{-n}}\to 0\]
   which gives  the following long exact sequence in homology \cite[III.\S6]{Brown1982}:
   \begin{multline*}
    \ldots\to H_1(\Gamma;\MM_1^{-}[\mathbf Z]) \xrightarrow{(t-\lambda^{-n})\cdot}
    H_1(\Gamma;\MM_1^{-}[\mathbf Z]) \rightarrow 
    H_1(  \Gamma;\MM^{-}_{\lambda^{-n}})  \xrightarrow{\partial\; } \\
    H_0(\Gamma;\MM_1^{-}[\mathbf Z]) \xrightarrow{(t-\lambda^{-n})\cdot }
    H_0(\Gamma;\MM_1^{-}[\mathbf Z]) \rightarrow 
    H_0(  \Gamma;\MM^{-}_{\lambda^{-n}})  \rightarrow 0\,.
   \end{multline*}
   Thus $\Delta_1^-(\lambda^{-n})=0$ if and only if $ \ker \partial$ is non-trivial.
   
   Next we claim that 
$ \ker \partial$ is non-trivial   if and only if  $\dim H^1(\Gamma ; \MM^+_{\lambda^n})> \dim H^0(\Gamma ; \MM^+_{\lambda^n})$. 
It follows from Lemma~\ref{prop:dualityDelta0} and Equation~\eqref{eqn:V[Z]}  that
the $\mathbf C[\mathbf Z]$-module  $H_0(\Gamma;\MM_1^{-}[\mathbf Z])$ is torsion, i.e.\ it is a finite
dimensional  $\mathbf C$-vector space.
 Hence,
 by exactness,
 $\operatorname{rank} \partial =\dim H_0(\Gamma;\MM^-_{\lambda^{-n}} )$ and
 \begin{align*}
 \dim \ker \partial&= \dim  H_1(\Gamma ; \MM^-_{\lambda^{-n}}) -  \operatorname{rank} \partial
 \\
 &= \dim   H_1(\Gamma ; \MM^-_{\lambda^{-n}}) - \dim  H_0(\Gamma ; \MM^-_{\lambda^{-n}})
 \\ 
 &= \dim   H^1(\Gamma ; \MM^+_{\lambda^{n}}) - \dim  H^0(\Gamma ; \MM^+_{\lambda^{n}}),
  \end{align*}
 by Kronecker duality~\eqref{eqn:Kronecker}, which proves the claim.
  Of course the same proof applies by symmetry for the opposite signs $\pm$ and $\mp$.
 \end{proof}
 
 \begin{proof}[Proof of Theorem~\ref{thm:nec}]
By Lemmas~\ref{lem:dimh1-dimh0} and \ref{lem:delta10} we get that, if
$\rho_\lambda$ can be deformed to irreducible
representations, then $\Delta_1^+(\lambda^n)=0$ or
$\Delta_1^-(\lambda^{-n})=0$.  
Corollary~\ref{cor:Alex-sym} yields
$\Delta_1^+(\lambda^n)=\Delta_1^-(\lambda^{-n})=0$.
 \end{proof}

 \begin{remark}\label{rem:bothinequalities} Notice that in the situation of Theorem~\ref{thm:nec}, from
   Lemma~\ref{lem:delta10} and Corollary~\ref{cor:Alex-sym} we get both inequalities  
  $$\dim H^1(\Gamma ; \MM^\pm_{\lambda^{\pm n}})> \dim H^0(\Gamma ; \MM^\pm_{\lambda^{\pm n}}).
  $$
\end{remark}

   We will later need the following construction. 
  Given a 1-cochain $c\in C^1(\Gamma;\MM^+_{\lambda^n})$, i.e.\  a map 
$c\co \Gamma\to  \MM^+_{\lambda^n}$,
consider the  map 
$\rho_\lambda^c\co\Gamma\to \mathrm{SL}_n(\mathbf C)$ given by
\begin{equation}
 \label{eqn:prerep}
\rho_\lambda^c(\gamma)=
\begin{pmatrix}
	  \operatorname{Id}_a & c (\gamma) \\
	  0 & \operatorname{Id}_b
       \end{pmatrix}
       \rho_\lambda(\gamma),\qquad\forall\gamma\in\Gamma.
\end{equation}

\begin{lemma}
\label{lem:repcocycle}
The map $\rho_\lambda^c\co\Gamma\to \mathrm{SL}_n(\mathbf C)$ given by
\eqref{eqn:prerep} is a representation  if and only if $c$ is a cocycle, i.e.\
$c\in Z^1(\Gamma;   \MM^+_{\lambda^n})$.

In addition, assuming that it is a representation,  $\rho_\lambda^c$ is
conjugate to $\rho_\lambda$ if and only if $c$ is a coboundary (i.e.\ $c\in
B^1(\Gamma;   \MM^+_{\lambda^n})$), 
if and only if $\rho_\lambda^c$ is completely reducible.           
\end{lemma}

\begin{proof}
The equivalence between being a representation and the cocycle condition is a
straightforward computation, as well as the equivalence between being conjugate
to $\rho_\lambda$
 and the coboundary condition. The equivalence with complete reducibility comes
from the fact that there is a unique orbit of completely reducible representations in the
fibre of the 
map $R(\Gamma, \mathrm{SL}_n(\mathbf C))\to X(\Gamma, \mathrm{SL}_n(\mathbf C))$
\cite{LM85}.
Hence two completely reducible representations having the same character are
conjugates.
\end{proof}

The following corollary generalizes a result of G.~Burde \cite{Burde1967} and
G.~de Rham \cite{dR1967}:

\begin{cor}\label{cor:Burde-deRham}
There exists a reducible, not completely reducible representation
$\rho^c_\lambda\co\Gamma\to\mathrm{SL}_n(\mathbf C)$ such that 
$\chi_{\rho^c_\lambda}=\chi_{\rho_\lambda}$ 
if and only if $\lambda^n$ is a root of the product of twisted Alexander
polynomials $\Delta_1^+(t) \Delta_0^+(t)$.
\end{cor}
\begin{proof}
By Lemma~\ref{lem:repcocycle}, such a representation exists if and only if 
$H^1(\Gamma;   \MM^+_{\lambda^n})$ or $H^1(\Gamma;   \MM^-_{\lambda^{-n}})$ does
not vanish.
By Kronecker duality this is equivalent to saying that  $H_1(\Gamma;  
\MM^+_{\lambda^n})$ or $H_1(\Gamma;   \MM^-_{\lambda^{-n}})$ does not vanish.
Then, the long exact sequence in the proof of Lemma~\ref{lem:delta10} shows that
this is equivalent to one of
$ H_1(\Gamma;\MM_1^{\pm }[\mathbf Z])$ or $ H_0(\Gamma;\MM_1^{\pm }[\mathbf Z])$
to have $(t-\lambda^{\pm n})$-torsion. With the duality of polynomials, Corollary~\ref{cor:Alex-sym}, this proves the lemma.
\end{proof}   
 
This corollary also applies when  $\alpha=\beta=1$ and $\lambda= \pm 1$. Since $\Delta_0(t)=(t-1)$, the vanishing $\Delta_0((\pm 1)^2)=0$  corresponds to the representations
 $$
 \gamma\mapsto \pm \begin{pmatrix}
                    1 &  d(\gamma)\\
                    0 & 1
                   \end{pmatrix},
 $$
 $\forall\gamma\in\Gamma$,
 where $d\co\Gamma\to(\mathbf C,+)$ is any group morphism.

 \section{Infinitesimal deformations and cup products}
\label{sec:inf_def_cup}

Throughout this section and the next one we assume the hypothesis of
Theorem~\ref{thm:suf},
namely that $\Delta_0^+(\lambda^n)\neq 0$  and  $\lambda^n$ is a simple root of 
$\Delta_1^+$.
By Corollary~\ref{cor:Alex-sym}, we also have that $\Delta_0^-(\lambda^{-n})\neq
0$  and  $\lambda^{-n}$ is a simple root of  $\Delta_1^-$.
Thus 
 the $\mathbf C[\mathbf Z]$-module
$H_0(\Gamma ; \MM^\pm_1[\mathbf Z])$ has no $(t-\lambda^{\pm n})$-torsion and
$H_1(\Gamma ; \MM^\pm_1[\mathbf Z])$  has a single $\mathbf C[t^{\pm 1}]/(t-
\lambda^{\pm n})$-factor. Furthermore, the following proposition gives more details on the cohomology.

\begin{prop}\label{prop:dim-H-M-lambda}
If $\Delta^+_0(\lambda^{n})\neq 0$ and $\lambda^n$ is a simple root of
$\Delta_1^+$, 
  then
  \begin{itemize}
   \item[(i)] \(
  \dim H^i(\Gamma ; \MM^\pm_{\lambda^{\pm n}})=
  \begin{cases}
  1 & \text{ if $i=1,2$,}\\
  0 & \text{ otherwise.}
  \end{cases}
  \)
  \item[(ii)] The $(t-\lambda^{\pm n})$-torsion of 
 $H^q(\Gamma; \MM^\pm_1[\mathbf Z])$ is zero for $q\neq 2$ and cyclic of the form 
 $\mathbf C[\mathbf Z]/(t- \lambda^{\pm n})$ for $q=2$.
  \end{itemize}

\end{prop}

\begin{proof}
%
In order to prove  the first assertion, we use  the long exact sequence  in the proof of
Lemma~\ref{lem:delta10}. The hypothesis on the twisted Alexander polynomials
gives that 
the $(t-\lambda^{\pm n})$-torsion of 
$H_i(\Gamma,  \MM_1^\pm[\mathbf Z])$ is zero for $i\neq 1$ and $t-\lambda^{\pm n}$ for
$i=1$. The long exact sequence gives
  \[
  \dim H_i(X ; \MM^\pm_{\lambda^{\pm n}})=
  \begin{cases}
  1 & \text{ if $i=1,2$,}\\
  0 & \text{ otherwise.}
  \end{cases}
  \]
 Hence the first assertion follows from Kronecker duality \eqref{eqn:Kronecker}.
 
For the second assertion, we use
the universal coefficient theorem for cohomology: for  any
{representation $\rho\co\Gamma\to\mathrm{GL}(V)$ we have:}
\begin{align}\label{eq:universal-coeff}
\overline H^q(X; V^*[\mathbf Z]) \cong & \hom_{\mathbf{C}[\mathbf{Z}]}\big( H_q(X;
V[\mathbf{Z}]) , \mathbf{C}[\mathbf{Z}]\big) \notag\\  & \qquad \oplus 
 \mathrm{Ext}_{\mathbf{C}[\mathbf{Z}]}\big( H_{q-1}(X; V[\mathbf{Z}]) ,
\mathbf{C}[\mathbf{Z}] \big)
\end{align}
where $\overline H^q(X; V^*[\mathbf Z])$ denotes the group $H^q(X; V^*[\mathbf Z])$
with the conjugate
$\mathbf{C}[\mathbf{Z}]$-module structure. For a detailed argument see
\cite[p.~638-639]{KirkLivingston99}. 
We apply \eqref{eq:universal-coeff} to the representation 
$\alpha\otimes \beta^*$ and its dual $\beta\otimes\alpha^*$.
The hypothesis on the twisted Alexander polynomials gives
that 
the $(t-\lambda^{\pm n})$-torsion of 
$H_i(\Gamma,  \MM_1^\pm[\mathbf Z])$ is zero for $i\neq 1$ and $t-\lambda^{\pm n}$ for
$i=1$.
Notice that $H_q(X; \MM^+_{1}[\mathbf{Z}])$ are torsion $\mathbf{C}[\mathbf{Z}]$-modules and,
regarding the conjugation of the action, 
$-\lambda^{-n}t(t^{-1}-\lambda^n)=(t-\lambda^{-n})$.
\end{proof}

From now on we   fix cocycles $$d_\pm\in Z^1(\Gamma;\MM^\pm_{\lambda^{\pm n}})$$ whose cohomology class do  not vanish.
Because $H^1(\Gamma;\MM^\pm_{\lambda^{\pm n}})\cong\mathbf{C}$, the elements
$d_\pm$ are unique up to adding a coboundary and up to multiplying by a non-zero scalar. Our next goal is to show that the cohomology class
of the cup product 
$\varphi\mathop{\smallcup} d_\pm$ does not vanish in $H^2(\Gamma; \MM^\pm_{\lambda^{\pm n}} )$. For that purpose we shall use the dual numbers.

\paragraph{Dual numbers.}
The \emph{algebra of dual numbers} is defined to be
$$
\mathbf C_{\varepsilon}= \mathbf C[\varepsilon ]/\varepsilon^2.
$$

Similarly define $\mathbf{C}_\varepsilon[\mathbf Z] = \mathbf{C}[\mathbf Z] \otimes_\mathbf{C} \mathbf{C}_\varepsilon$
and
\[ 
\MM^\pm_{\lambda^{\pm n}(1\pm\varepsilon)} 
= \big(\MM^\pm_1[\mathbf Z]  \otimes_{\mathbf C} \mathbf{C}_\varepsilon\big)
/\big(t-\lambda^{\pm n}(1\pm\varepsilon)\big) \,.
\]
 \begin{lemma}\label{lem:dim-H-epsilon}
  If  $\lambda^n\in \mathbf C^*$ is a simple root of $\Delta_1^+$ such that $\Delta^+_0(\lambda^{n})\neq 0$, 
  then $  \dim H^1(\Gamma ; \MM^\pm_{\lambda^{\pm n}(1\pm\varepsilon)})=1$.
   \end{lemma}

\begin{proof}
Notice that $\Delta^-_0(\lambda^{-n})\neq 0$ and $\lambda^{-n}$ is a simple root of $\Delta_1^-$
by Corollary~\ref{cor:Alex-sym}.

 We have that $H^1(\Gamma ; \MM^+_1[\mathbf Z]\otimes \mathbf C_{\varepsilon})\cong H^1(\Gamma ; \MM^+_1[\mathbf Z])\otimes \mathbf C_{\varepsilon}$ since $\mathbf C_{\varepsilon}$ is a trivial $\Gamma$-module (isomorphic to $\mathbf C^2$).
As before, the short exact sequence
 \[0 \to \MM^+_1[\mathbf Z]  \otimes_{\mathbf C} \mathbf{C}_\varepsilon
 \xrightarrow{(t-\lambda^{n}(1+\varepsilon)) \cdot }
\MM^+_1[\mathbf Z]  \otimes_{\mathbf C} \mathbf{C}_\varepsilon \to \MM^+_{\lambda^{n}(1+\varepsilon)}\to 0
 \]
 gives a long exact sequence in cohomology (see \cite[III.\S6]{Brown1982}), 
 \begin{multline*}
 \cdots\to
 H^i(\Gamma ; \MM^+_1[\mathbf Z])\otimes \mathbf C_{\varepsilon}
 \xrightarrow{(t-(\lambda^{n}(1+\varepsilon)) \cdot }
 H^i(\Gamma ; \MM^+_1[\mathbf Z])\otimes \mathbf C_{\varepsilon}\\
 \to
 H^i(\Gamma ; \MM^+_{\lambda^{n}(1+\varepsilon)} )\to 
  H^{i+1}(\Gamma ; \MM^+_1[\mathbf Z])\otimes \mathbf C_{\varepsilon}\to\cdots.
\end{multline*}
Note that for $\mu\neq\lambda^n$, $\mu\in\mathbf C$,  multiplication by $t-\lambda^{n}(1+\varepsilon)$ induces an automorphism of
 $\mathbf{C}_\varepsilon[\mathbf Z] / (t-\mu)^k$.
Therefore, we are interested in the $(t-\lambda^n)$-torsion of $H^q(\Gamma ; \MM^+_1[\mathbf Z])$
described by Proposition~\ref{prop:dim-H-M-lambda}:
it vanishes for $q\neq 2$ and it is $\mathbf C[\mathbf Z]/(t- \lambda^{n})$ for $q=2$. 
Hence, multiplication by $(t-\lambda^{n}(1+\varepsilon))$ on $H^i(\Gamma ;
\MM^+_1[\mathbf Z]\otimes \mathbf C_\varepsilon)$ is an isomorphism for
$i\neq2$. In order to understand the effect of the multiplication on
$H^2(\Gamma ; \MM^+_1[\mathbf Z]\otimes \mathbf C_\varepsilon)$ it is sufficient
to consider 
 multiplication by $(t-\lambda^{n}(1+\varepsilon))$ on
 $$
 \mathbf C[\mathbf Z]/(t- \lambda^n)\otimes\mathbf C_\varepsilon \cong   \mathbf
C[\mathbf Z]/(t- \lambda^n) \oplus  \varepsilon\, \mathbf C[\mathbf Z]/(t-
\lambda^n) .
 $$

 Since $t-\lambda^n$ vanishes in this ring, multiplication by $(t-\lambda^{n}(1+\varepsilon))$ is equivalent to multiplication by $-\varepsilon\lambda^n$ on 
 $\mathbf C_\varepsilon\cong \mathbf{C}\oplus\varepsilon\,\mathbf{C} $. Therefore, its kernel and cokernel have $\mathbf C$-dimension $1$, which proves 
$\dim_{\mathbf C} H^1(\Gamma ; \MM^+_{\lambda^{n}(1+\varepsilon)}) =1$. 

By symmetry the same argument yields $\dim_{\mathbf C} H^1(\Gamma ; \MM^-_{\lambda^{-n}(1-\varepsilon)}) =1$. 
\end{proof}

\paragraph{Cup product and Bockstein homomorphism.} Let $A_{1}$, $A_{2}$ and $A_{3}$ be $\Gamma$-modules.
The cup product of two cochains $c_i \in C^1(\Gamma;A_{i})$, $i=1,2$ is the cochain
$c_1 \smallcup c_2 \in C^{2}(\Gamma;A_1\otimes A_2)$ defined by
\begin{equation}\label{eq:cupprod}
c_1 \smallcup c_2(\gamma_1,\gamma_2) := c_1(\gamma_1) \otimes
    \gamma_1\,c_2(\gamma_{2})\,.
\end{equation}
 Here $A_{1}\otimes A_{2}$ is a $\Gamma$-module via the diagonal
action.

It is possible to combine the cup product with any $\Gamma$-invariant, bilinear map
$b \co A_1\otimes A_2 \to A_3$.
So we obtain a cup product
$$
\sideset{_b}{}{\mathop{\smallcup}} \co
     C^1(\Gamma;A_{1}) \otimes C^1(\Gamma;A_{2}) \xrightarrow{\mathop{\smallcup}}
                  C^1(\Gamma;A_1\otimes A_2)\xrightarrow{b}C^{2}(\Gamma;A_3)\,. $$
                  For details see \cite[V.3]{Brown1982}.
In what follows we are mainly interested in the case where the bilinear form is simply the matrix multiplication, i.e.
\[ 
\mathbf C \otimes \MM^\pm_{\lambda^{\pm n}} \to \MM^\pm_{\lambda^{\pm n}}
\quad \text{ or } \quad
\mathfrak{sl}_a(\mathbf C)\otimes \MM^+_{\lambda^n} \to \MM^+_{\lambda^n}\,.
\]
Hence we will write simply ``$\smallcup$'' for such a cup-product when no confusion can arise.

     Let $b\co A_{1}\otimes A_{2}\to A_{3}$ be bilinear  and let
     $\tau\co A_{2}\otimes A_{1}\to A_{1}\otimes A_{2}$ be the
twist operator. Then for
        $c_i\in C^1(\Gamma;A_{i})$, $i=1,2$, we define
       the cup-product
       \[ \sideset{_{b\circ\tau}}{}{\mathop{\smallcup}}\co 
       C^1(\Gamma;A_{2}) \otimes C^1(\Gamma;A_{1}) \to
                  C^{2}(\Gamma;A_3)\,.\]
 Again we are mainly interested in matrix multiplication and we will write simply 
 ``$\sideset{_\tau}{}{\mathop{\smallcup}}$'' for such a cup-product when no confusion can arise.
 
 \begin{example}
 Let $c_a\in C^1(\Gamma ;  \mathfrak{sl}_a(\mathbf C))$ and 
 $d\in C^1(\Gamma ;  \MM^+_{\lambda^n})$ be given.
 Then 
 \begin{align*} 
 c_a\mathop{\smallcup} d (\gamma_1,\gamma_2) &= c_a(\gamma_1)\, \gamma_1 d(\gamma_2)
 =  \lambda^{n\varphi(\gamma_1)} c_a(\gamma_1) \alpha(\gamma_1) d(\gamma_2) \beta(\gamma_1)^{-1} 
  \intertext{and}
  d\sideset{_\tau}{}{\mathop{\smallcup}} c_a (\gamma_1,\gamma_2) &=
 \gamma_1 c_a(\gamma_2)\, d(\gamma_1)  =
 \alpha(\gamma_1) c_a (\gamma_2) \alpha(\gamma_1)^{-1} \, d(\gamma_1)\,.
 \end{align*}
 \end{example}
\begin{remark}\label{rem:notwist}
Note that if $z_a \in Z^1(\Gamma ;  \mathfrak{sl}_a(\mathbf C))$ and 
 $d_+\in Z^1(\Gamma ;  \MM^+_{\lambda^n})$ are cocycles then for $f\co\Gamma\to \MM^+_{\lambda^n}$
 given by $f(\gamma)= z_a(\gamma)\, d_+(\gamma)$ we have
 \[
 \delta f (\gamma_1,\gamma_2) +  z_a\mathop{\smallcup} d_+ (\gamma_1,\gamma_2) +
  d_+\sideset{_\tau}{}{\mathop{\smallcup}} z_a (\gamma_1,\gamma_2) =0
 \]
 i.e.\ $d_+\sideset{_\tau}{}{\mathop{\smallcup}} z_a\sim - z_a\mathop{\smallcup} d$
 in $C^2(\Gamma;\MM^+_{\lambda^n})$.
\end{remark}
 \begin{lemma}
   \label{lem:bockstein}
Consider the non-split exact sequence of $\Gamma$-modules
 $$
 0\to\MM^\pm_{\lambda^{\pm n}} \xrightarrow{\varepsilon\cdot} 
 \MM^\pm_{\lambda^{\pm n}(1\pm\varepsilon)} \longrightarrow 
 \MM^\pm_{\lambda^{\pm n}} \to 0\,.
 $$
 Then the image of the cohomology class represented by 
 $d_\pm$ (in  $H^1(\Gamma ;  \MM^\pm_{\lambda^{\pm n}})$) under the Bockstein
homomorphism is represented by the cup product $d_\pm\mathop{\smallcup}\varphi$
(in $H^2(\Gamma ;  \MM^\pm_{\lambda^{\pm n}})$).
   \end{lemma}
\begin{proof}
In order to calculate the Bockstein homomorphism
$\bock \co H^1(\Gamma ;  \MM^+_{\lambda^n})\to H^2(\Gamma ;  \MM^+_{\lambda^n})$
we proceed as follows (according to the snake lemma): given a cocycle $d_+ \in
Z^1(\Gamma ;  \MM^+_{\lambda^n})$  we choose a cochain
$\tilde d_+ \in  C^1(\Gamma; \MM^+_{\lambda^n(1+\varepsilon)})$ which projects
onto $d_+$
and then we calculate $\delta_\varepsilon \tilde d_+\in C^2(\Gamma ;
\MM^+_{\lambda^n(1+\varepsilon)})$ where $\delta_\varepsilon$ denotes the
coboundary operator of 
$C^*(\Gamma ; \MM^+_{\lambda^n(1+\varepsilon)})$.
Since $d_+$ is a cocycle we obtain $ \delta_\varepsilon \tilde d_+ =
\varepsilon\cdot z$ for a 2-cocycle
$z\in Z^2(\Gamma ;  \MM^+_{\lambda^n})$,
which represents the image of the Bockstein map.
By abusing notation, we also denote the map constructed in this way by
$\bock \co Z^1(\Gamma ;  \MM^+_{\lambda^n})\to Z^2(\Gamma ;  \MM^+_{\lambda^n})$,
even if it is only well defined in cohomology. In particular
$\bock(d_+)  \sim z$. 
In order to calculate  $ z \in Z^2(\Gamma ; \MM^{+}_{\lambda^n})$ 
we choose
$\tilde d_+ = d_+ +\varepsilon\, 0$:
\begin{align*}
\delta_\varepsilon \tilde d_+ (\gamma_1,\gamma_2) &=
\gamma_1\, \tilde d_+(\gamma_2) - \tilde d_+(\gamma_1\gamma_2) + \tilde d_+(\gamma_1)\\
&=\lambda^{ n\varphi(\gamma_1)}(1+\varepsilon\varphi(\gamma_1)) \alpha(\gamma_1)
d_+(\gamma_2) \beta(\gamma_1)^{-1}- d_+(\gamma_1\gamma_2) + d_+(\gamma_1)\\
&= \varepsilon \, \varphi(\gamma_1) \,\gamma_1 d_+(\gamma_2) = 
\varepsilon\cdot \varphi\mathop{\smallcup} d_+(\gamma_1,\gamma_2)\,.
\end{align*}
Therefore, $\bock (d_+)\sim  \varphi \smallcup d_+$. The
calculation for 
$\bock (d_-)\sim \varphi\smallcup d_-$ is similar.
\end{proof}

   \begin{cor}
   \label{cor:productnonzero}
    Assume $\dim_{\mathbf C} H^1(\Gamma;\MM^{\pm}_{\lambda^{\pm n}})=1$,
    $H^0(\Gamma;\MM^{\pm}_{\lambda^{\pm n}})=0$ and let $d_\pm\in Z^1(\Gamma; \MM^\pm_{\lambda^{\pm n}})$ be not cohomologous to zero. 
    
    Then 
    $\dim H^1(\Gamma;\MM^\pm _{\lambda^{\pm n}(1\pm\varepsilon)})=1$ iff the cup product $\varphi\smallcup d_{\pm}$ does not vanish in $H^2(\Gamma ; \MM^\pm_{\lambda^{\pm n}})$. 
   \end{cor}
   
\begin{proof}
 Consider the non-split exact sequence of $\Gamma$-modules
\[
 0\to\MM^\pm_{\lambda^{\pm n}} \xrightarrow{\varepsilon\cdot} 
 \MM^\pm_{\lambda^{\pm n}(1\pm\varepsilon)} \longrightarrow 
 \MM^\pm_{\lambda^{\pm n}} \to 0
 \]
 and the corresponding long exact sequence in cohomology:
 \begin{multline}\label{eqn:suite}
 0\to H^1(\Gamma ; \MM^\pm_{\lambda^{\pm n}}) \xrightarrow{\varepsilon\cdot}
 H^1(\Gamma ;  \MM^\pm_{\lambda^{\pm n}(1\pm\varepsilon)}) \to \\
  H^1(\Gamma ;  \MM^\pm_{\lambda^{\pm n}}) \xrightarrow\bock  H^2(\Gamma ; \MM^\pm_{\lambda^{\pm n}})\,.
\end{multline}

 The lemma then follows from the sequence \eqref{eqn:suite}, as by Lemma~\ref{lem:bockstein}
 $\bock(d_\pm)\sim \varphi\mathop{\smallcup}d_\pm$ and
 $\dim H^1(\Gamma ; \MM^\pm_{\lambda^{\pm n}})=1$.
 \end{proof}
 Combining  Proposition~\ref{prop:dim-H-M-lambda}, Lemma~\ref{lem:dim-H-epsilon}, and Corollary~\ref{cor:productnonzero}, we deduce:
 \begin{cor}
  \label{cor:cupnonzero}
  Under the hypothesis of Theorem~\ref{thm:suf}, the cup product $\varphi\smallcup d_{\pm}$ does not vanish in 
  $H^2(\Gamma ; \MM^\pm_{\lambda^{\pm n}})$. 
 \end{cor}

 \section{A non completely reducible representation $\rho^+$}
 \label{sec:regrho}
 
 In this section we construct $\rho^+\in R(\Gamma,\mathrm{SL}_n(\mathbf{C})$ with the same character as $\rho_\lambda$ 
 but that is not completely reducible. We show that $\rho^+$ is a smooth point of $R(\Gamma,\mathrm{SL}_n(\mathbf{C}))$ and that it can be deformed
 to irreducible representations. This proves Theorem~\ref{thm:suf}, because the orbit by conjugation of $\rho^+$ accumulates to $\rho_ \lambda$.

Assume throughout this section that the hypothesis of Theorem~\ref{thm:suf}  hold true. Namely (using Corollary~\ref{cor:Alex-sym}),
$\Delta_0^{\pm n}(\lambda^{\pm n})\neq 0$ and $\lambda^{\pm n}$ is a simple root of $\Delta_1^\pm$.    
Recall we have fixed $d_+\in Z^1(\Gamma ; \MM^+_{\lambda^n})$ a cocycle not homologous to zero. 
Let 
$$
\rho^+=\begin{pmatrix}
	  \operatorname{Id}_a & d_+ \\
	  0 & \operatorname{Id}_b
       \end{pmatrix}
       \rho_\lambda.
$$
By Lemma~\ref{lem:repcocycle} $\rho^+$ is not completely reducible, hence it is not conjugate to $\rho_\lambda$, even if it has the same character.
We shall prove that $\rho^+$ is a regular point of $R(\Gamma , \mathrm{SL}_n(\mathbf C))$ and that the local dimension is
$\dim \mathrm{SL}_n(\mathbf C) + n-1= n^2+n-2$. Then we will argue that the reducible representations around $\rho_\lambda$ form a Zariski
closed algebraic set of dimension $n^2+n-3$, which will prove Theorem~\ref{thm:suf}.

Let $P^+=\left(
	\begin{smallmatrix}
	 * & * \\
	 0 & *
	\end{smallmatrix}\right)
\subset \mathrm{SL}_n(\mathbf C)$ be the maximal parabolic subgroup that preserves $\mathbf C^a\oplus 0$. Its Lie algebra is denoted by $\mathfrak{p}^+\subset \mathfrak{sl}_n(\mathbf C)$.
We have two short exact sequences of $\Gamma$-modules via the action of $\Ad\rho^+$:
\begin{equation}
\label{eqn:exactp} 
0\to\MM^+_{\lambda^n}\to\mathfrak{p}^+\to \mathcal{D}\to 0,
\end{equation}
where 
\begin{equation}
 \label{eqn:D}
 \mathcal{D}=\mathfrak{sl}_a(\mathbf C )\oplus\mathfrak{sl}_b(\mathbf C )\oplus\mathbf C ,
\end{equation}
and
\begin{equation}
\label{eqn:exactsl}
0\to\mathfrak{p}^+\to \mathfrak{sl}_n(\mathbf C) \to \MM^-_{\lambda^{-n}}\to 0 .
\end{equation}
We will use the corresponding long exact sequences in cohomology to compute $H^1(\Gamma ; \mathfrak{sl}_n(\mathbf C)_{\Ad\rho^+})$. 
The first step is the following lemma.

\begin{lemma}
 \label{lem:h0p}
 $H^0(\Gamma ;  \mathfrak{p}^+)=0$.
\end{lemma}

\begin{proof}
The long exact sequence 
 associated to \eqref{eqn:exactp} starts with
 \[
 0= H^0(\Gamma ;  \MM^+_{\lambda^n})\to H^0(\Gamma ;  \mathfrak{p}^+)\to
 H^0(\Gamma ;  \mathcal D)\xrightarrow{\bock}H^1(\Gamma ;  \MM^+_{\lambda^n})\,.
 \]
The group $H^0(\Gamma ;   \mathcal{D})\cong \mathbf C$ is generated by
the invariant element 
 $
 \left(
	\begin{smallmatrix}
	 - b\textrm{Id}_a& 0 \\
	 0 & a \textrm{Id}_b
	\end{smallmatrix}\right)\in \mathcal{D}^\Gamma
 $. A similar calculation as in the proof of Lemma~\ref{lem:bockstein} using the snake lemma gives
 $\bock \left(
	\begin{smallmatrix}
	 - b\textrm{Id}_a& 0 \\
	 0 & a \textrm{Id}_b
	\end{smallmatrix}\right) \sim n\,d_+$.
	Therefore $\bock\co  H^0(\Gamma ;  \mathcal D)\to H^1(\Gamma ;  \MM^+_{\lambda^n})$
	is injective and hence $H^0(\Gamma ;  \mathfrak{p}^+)=0$.
\end{proof}

 We continue the long exact sequence in cohomology associated to \eqref{eqn:exactp}:
$$
0\to\mathbf C\to H^1(\Gamma ; \MM^+_{\lambda^n})\to  H^1(\Gamma ; \mathfrak{p}^+)\to H^1(\Gamma ; \mathcal{D})\xrightarrow{\bock}H^2(\Gamma ; \MM^+_{\lambda^n}).
$$
Since $H^i(\Gamma ; \MM^+_{\lambda^n})\cong \mathbf C$ for $i=1,2$ by Proposition~\ref{prop:dim-H-M-lambda}, it shortens to
$$
0\to H^1(\Gamma ; \mathfrak{p}^+)\to H^1(\Gamma ; \mathcal{D})\xrightarrow{\bock} H^2(\Gamma ; \MM^+_{\lambda^n}).
$$

Next we aim to compute $\bock\co H^1(\Gamma ; \mathcal{D})\to H^2(\Gamma ; \MM^+_{\lambda^n})$. For this we use the decomposition \eqref{eqn:D}.
Every element in $H^1(\Gamma ; \mathcal{D})$ is represented by a cocycle 
\begin{equation}
 \label{eqn:diagonalcocycle}
 \vartheta=
\begin{pmatrix}
 z_a & 0\\
 0 & z_b
\end{pmatrix}
+z \varphi 
\begin{pmatrix}
    -b \operatorname{Id}_a & 0 \\
    0 & a \operatorname{Id}_b
\end{pmatrix}
\end{equation}
where $z_a\in Z^1(\Gamma ;  \mathfrak{sl}_a(\mathbf{C}))$,
$z_b\in Z^1(\Gamma ;  \mathfrak{sl}_b(\mathbf{C}))$, and $z\in \mathbf C$.

\begin{lemma}
 \label{lemma:connecting} 
 For a cocycle $\vartheta\in Z^1(\Gamma ; \mathcal{D})$ as in \eqref{eqn:diagonalcocycle}, 
 $$
 \bock(\vartheta) \sim z_a \smallcup d_+ + d_+\smallcup z_b + z \, n\, d_+\smallcup\varphi\,.
 $$
\end{lemma}
\begin{proof}
As in Lemma~\ref{lem:bockstein} we compute $\bock(\vartheta)$ by using the snake lemma.
Namely, let $\delta^+$ be the coboundary operator of $C^*(\Gamma ; \mathfrak p^+)$, and
let $\tilde\vartheta\in C^1(\Gamma ; \mathfrak p^+)$ be the composition of $\vartheta$ with the inclusion 
$\mathcal{D}\hookrightarrow\mathfrak{p}^+$. Then
\[
\delta^+\, \tilde\vartheta (\gamma_1,\gamma_2)=
\begin{pmatrix}
0 & -\gamma_1 z_a(\gamma_2)\, d_+(\gamma_1) + d_+(\gamma_1)\, \gamma_1z_b(\gamma_2) + 
z\, n\, d_+(\gamma_1)\,\varphi(\gamma_2)\\
0 & 0
\end{pmatrix}
\]
and hence $ \bock(\vartheta) \sim
-d_+ \sideset{_\tau}{}{\smallcup} z_a+ d_+\smallcup z_b+ z \, n\, d_+\smallcup\varphi$.

Finally, Remark~\ref{rem:notwist} proves the lemma.
\end{proof}

 Since $\varphi\smallcup d_{\pm}$ is not cohomologous to zero by Corollary~\ref{cor:cupnonzero},
and since $ H^2(\Gamma ; \MM^+_{\lambda^n})\cong\mathbf C$ by Proposition~\ref{prop:dim-H-M-lambda}, we deduce:

\begin{cor}
\label{cor:h1p+}
 The cohomology group $H^1(\Gamma ; \mathfrak{p}^+)\cong \mathbf C^{n-2}$ is
naturally identified to the kernel of the rank one map:
$$
H^1(\Gamma ; \mathcal D)\cong H^1(\Gamma ; \mathfrak{sl}_a(\mathbf C)) \oplus H^1(\Gamma ; \mathfrak{sl}_b(\mathbf C)) \oplus \mathbf C
\xrightarrow{\bock} H^2(\Gamma ; \MM^+_{\lambda^n})\cong\mathbf C.
$$
\end{cor}

Next we consider the long exact sequence corresponding to \eqref{eqn:exactsl}:
\begin{equation}
 \label{eqn:exactpslm}
0\to H^1(\Gamma ; \mathfrak{p}^+)\to 
H^1(\Gamma ; \mathfrak{sl}_n(\mathbf C)_{\Ad\rho^+}) \to H^1(\Gamma ; \MM^-_{\lambda^{-n}}). 
\end{equation}
Hence 
\begin{equation}
\label{eqn:ineqdimsl}
\dim H^1(\Gamma ; \mathfrak{sl}_n(\mathbf C)_{\Ad\rho^+})\leq \dim H^1(\Gamma ; \mathfrak{p}^+) + \dim H^1(\Gamma ; \MM^-_{\lambda^{-n}})= n-2+1=n-1.
\end{equation}
On the other hand we  apply Poincar\' e duality to the long exact sequence of the pair $(X,\partial X)$ (see \eqref{eqn:pair})
and we obtain as in Equation~\eqref{eqn:dimH^1partialX}:
\begin{equation}\label{eqn:upper-bound-dimH1}
\dim H^1(\partial X;\mathfrak{sl}_n(\mathbf C)_{\Ad\rho^+}) \leq 2 \dim H^1(\Gamma ; \mathfrak{sl}_n(\mathbf C)_{\Ad\rho^+})\leq 2( n-1).
\end{equation}

\begin{prop}
 \label{prop:rho+smooth}
 $\rho^+$ is a regular point of $R(\Gamma,\mathrm{SL}_n(\mathbf C))$ of dimension $n^2+n-2$.
\end{prop}
\begin{proof}
The dimension inequality of Lemma~\ref{lem:H1-Z+Z} and the inequality \eqref{eqn:upper-bound-dimH1}
yield $\dim H^1(\Gamma;\mathfrak{sl}_n(\mathbf C)_{\Ad\rho^+})=n-1$, and we apply Proposition~\ref{prop:MichaelOuardia}.
\end{proof}

Before proving that the irreducible component of $R(\Gamma,\mathrm{SL}_n(\mathbf C))$ containing  $\rho^+$ also contains irreducible representations, we need a remark and two lemmas.

\begin{remark}
\label{rem:pexact}
 It follows from the proof of Proposition~\ref{prop:rho+smooth} that Inequalities~\eqref{eqn:ineqdimsl} and~\eqref{eqn:upper-bound-dimH1} are equalities,
 therefore  \eqref{eqn:exactpslm} becomes  a short exact sequence:
 \[
0\to H^1(\Gamma ; \mathfrak{p}^+)\to H^1(\Gamma ; \mathfrak{sl}_n(\mathbf C)_{\Ad\rho^+}) \to H^1(\Gamma ; \MM^-_{\lambda^{-n}})\to 0.
\]
 \end{remark}

\begin{lemma}
\label{lemma:p+smooth}
The representation  $\rho^+$ is a smooth point of $R(\Gamma, P^+)$.
 \end{lemma}

\begin{proof}
The key tool here is the vanishing of Goldman's obstructions to integrability
\cite{Goldman1984},
which relies on the naturality of these obstructions and the vanishing
for $\mathfrak{sl}_n(\mathbf C)$. This vanishing was not explicit in the proof
of Proposition~\ref{prop:rho+smooth}, because here
we used Proposition~\ref{prop:MichaelOuardia} 
from \cite{HM14}, but it is used in~\cite{HM14}.

By Remark~\ref{rem:pexact}, the long 
 exact sequence in cohomology associated to \eqref{eqn:exactsl} yields an
injection
$$
0\to H^2(\Gamma ; \mathfrak{p}^+)\to H^2(\Gamma ; \mathfrak{sl}_n(\mathbf C)_{\Ad\rho^+}).
$$
Now Goldman's obstructions to integrability are natural for the inclusion
 $\mathfrak{p}^+\to \mathfrak{sl}_n(\mathbf C)$. In addition, the obstructions
of a cocycle in $\mathfrak{p}^+$ remain in  $\mathfrak{p}^+$, 
  because $\mathfrak{p}^+\to \mathfrak{sl}_n(\mathbf C)$ is a subalgebra (closed
under the Lie bracket) and a $\Gamma$-submodule of $ \mathfrak{sl}_n(\mathbf
C)$.
 Since $\rho^+$ is a smooth point of $R(\Gamma, \operatorname{SL}_n(\mathbf
C))$,
 for any cocycle in $Z^1(\Gamma;  \mathfrak{p}^+)$ the infinite sequence of
obstructions to integrability in $H^2(\Gamma ; \mathfrak{sl}_n(\mathbf C)_{\Ad\rho^+})$
vanish, so the infinite 
 sequence of obstructions to integrability in  $H^2(\Gamma ;    \mathfrak{p}^+ )$
also vanish. This establishes that any infinitesimal deformation is formally
integrable
 and it follows from Artin's theorem \cite{Artin1968} that it is actually
integrable, which proves the lemma.  
\end{proof}

Let $1\leq k\leq n$ and let $R_k\subset R(\Gamma,\mathrm{SL}_n(\mathbf C))$ denote the
subset of representations $\rho$ such that $\rho(\Gamma)$ preserves a $k$-dimensional subspace of $\mathbf C^n$.
\begin{lemma} \label{lem:invariant-subspace-closed}
For all $1\leq k\leq n$, the subset $R_k\subset R(\Gamma,\mathrm{SL}_n(\mathbf C))$ is Zariski-closed.
\end{lemma}
\begin{proof}
The assertion is clear for $k=n$ since $R_n= R(\Gamma,\mathrm{SL}_n(\mathbf C))$.
Hence suppose that $1\leq k<n$ and let 
$P(k)\subset \mathrm{SL}_n(\mathbf C)$ denote the parabolic subgroup which preserves 
$\mathbf C^k\times\{0\}\subset\mathbf C^n$. The set
$R(\Gamma,P(k))\subset R(\Gamma,\mathrm{SL}_n(\mathbf C))$ is Zariski-closed since it is 
given by a finite number of equations. Moreover, we have 
\[
R_k = \mathrm{SL}_n(\mathbf C)\cdot R(\Gamma,P(k)) = \mathrm{SL}_n(\mathbf C)/P(k) \cdot R(\Gamma,P(k))
\]
since $P(k)$ preserves $R(\Gamma,P(k))$. Finally, $R_k$
is Zariski-closed since the quotient $\mathrm{SL}_n(\mathbf C)/P(k) $ is complete
(see \cite[\S 0.15]{Humphreys1995}).
\end{proof}
 
 \begin{lemma}\label{lem:no-invariant-subspace}
The unique proper invariant subspace of $\rho^+(\Gamma)$ is $\mathbf C^a\times\{0\}$.
\end{lemma} 

\begin{proof}
We compute the possible non-zero invariant subspaces of $\rho^+(\Gamma)$ by taking a non-zero vector $v\in\mathbf C^n$
and considering the linear span of its orbit $\langle\rho^+(\Gamma) v\rangle$. When $v\in \mathbf C^a\times\{0\}$, then 
 $\langle\rho^+(\Gamma)v\rangle= \mathbf C^a\times\{0\}$ because $\alpha$ is irreducible. So  we assume that the projection of $v$ to the quotient
 $\mathbf C^n/ \mathbf C^a\times\{0\}$ does not vanish, and since $\beta$ is irreducible, the projection of the linear span $\langle\rho^+(\Gamma)v\rangle$
 is the whole $\mathbf C^n/ \mathbf C^a\times\{0\}$. In particular the dimension of $\langle\rho^+(\Gamma)v \rangle$ is at least $b$. Notice that
 $\dim_{\mathbf C} \langle\rho^+(\Gamma)v\rangle = b$ cannot occur, because this would yield a direct sum $
  \mathbf C^n=\mathbf C^a\times\{0\} \oplus \langle\rho^+(\Gamma)v\rangle$;  by Lemma~\ref{lem:repcocycle} this would contradict the choice of $\rho^+$
  and the non-triviality of the cohomology class of $d_ +$. Therefore 
   $\dim_{\mathbf C} \langle\rho^+(\Gamma)v \rangle > b$, so that 
 $\langle\rho^+(\Gamma) v \rangle$ contains at least a nontrivial vector in $\mathbf C^a\times\{0\}$ (the kernel of the projection). Irreducibility of $\alpha$ gives now 
   $\langle\rho^+(\Gamma)v\rangle=\mathbf C^n$.
\end{proof}

Let $S$ be the component of $R(\Gamma,\mathrm{SL}_n(\mathbf C))$ that contains $\rho^+$. In particular,
$\dim S= n^2+n-2$. We  show that $S$ contains irreducible representations.

\begin{prop}
  \label{prop:rho+irreducible}
 $S$ contains irreducible representations.
\end{prop}

\begin{proof} 
We prove the proposition by contradiction, hence assume that there is a
Zariski  neighborhood $U\subset S\subset R(\Gamma,\mathrm{SL}_2(\mathbf C))$ of $\rho^+$
so that all representations in $U$ are reducible.
By  Lemmas~\ref{lem:invariant-subspace-closed} and \ref{lem:no-invariant-subspace}, the choice of the $U$ can be
made so that the representations in $U$ have only an $a$-dimensional invariant subspace.

In particular every representation in $U$ is conjugate to a representation in $P^+=P(a)$. 
Therefore
given any Zariski neighborhood $U^+\subset R(\Gamma,P^+)$ of $\rho^+$, $U$ can be chosen
so that every representation in $U$ is conjugate to a representation in $U^+$.
As $\rho^+$ is a smooth point of $R(\Gamma, P^+)$  by
Lemma~\ref{lemma:p+smooth}, $\rho^+$ is contained in a single irreducible
component $S^+$ of $R(\Gamma, P^+)$, and we may chose $U^+\subset S^+$. 
This
yields the inclusion 
$$
U\subset \mathrm{SL}_n(\mathbf C)\cdot U^+\subset \mathrm{SL}_n(\mathbf C)\cdot
S^+.
$$
Now we reach the contradiction by computing dimensions. 
Using that  $P^+$ stabilizes $S^+$:
\begin{equation*}
\dim U \leq \dim (\mathrm{SL}_n(\mathbf C)\cdot S^+)\leq \dim
(\mathrm{SL}_n(\mathbf C)/P^+) + \dim S^+,\\
\end{equation*}
where  $ \dim (\mathrm{SL}_n(\mathbf C)/P^+) = n^2-1-\dim \mathfrak{p}^+$, and
\[
\dim S^+ = \dim H^1(\Gamma ; \mathfrak{p}^+)+\dim \mathfrak{p}^+- \dim
H^0(\Gamma ; \mathfrak{p}^+) = n-2+\dim \mathfrak{p}^+ -0.
\]
This yields
$
 \dim U\leq n^2+ n-3$, contradicting Proposition~\ref{prop:rho+smooth},
 which asserts that  $\dim U = \dim S= n^2+n-2$.
\end{proof}

\section{The neighborhood of $\chi_\lambda$}
\label{sec:nbhood}

The aim of this section is to prove Theorem~\ref{thm:str} i.e.\ we determine the local structure 
of the character variety $X(\Gamma,\mathrm{SL}_n(\mathbf
C))$ at  $\chi_\lambda$, the character of the representation $\rho_\lambda$ given by \eqref{eq:rho_lambda}.
For this
purpose we will identify the quadratic cone of $X(\Gamma,\mathrm{SL}_n(\mathbf
C))$  
at  $\chi_\lambda$  by means of algebraic
obstructions to integrability. Moreover, we will describe 
these obstructions geometrically.

Before discussing the components of the variety of characters, we need to discuss the
components of the variety of representations. 
In Section~\ref{sec:regrho} we
have constructed $S$ a component of 
$R(\Gamma,\mathrm{SL}_n(\mathbf C))$ of dimension $n^2+n-2$
that contains $\rho^+$ and irreducible representations 
(Proposition~\ref{prop:rho+smooth} and Proposition~\ref{prop:rho+irreducible}).

Next we discuss a component of reducible representations.
The representation variety $R(\Gamma,\mathrm{SL}_n(\mathbf C))$
contains
\[
R(\Gamma,\mathrm{SL}_a(\mathbf C))\times  R(\Gamma,\mathrm{SL}_b(\mathbf C))\times R(\Gamma, \mathbf C^*)
\]
where the inclusion is given by 
$$
(\alpha',\beta',\lambda')\mapsto ((\lambda')^{b\varphi} \otimes \alpha')\oplus 
( (\lambda')^{-a\varphi} )\otimes \beta').
$$
By our hypothesis on infinitesimal regularity, $\alpha\in R(\Gamma,\mathrm{SL}_a(\mathbf C))$ and
$\beta\in R(\Gamma,\mathrm{SL}_b(\mathbf C))$ are smooth points which are contained in a unique components $V_\alpha\subset R(\Gamma,\mathrm{SL}_a(\mathbf C))$ and
$V_\beta\subset R(\Gamma,\mathrm{SL}_b(\mathbf C))$ respectively.
Hence we obtain an embedding
$$
 V_\alpha \times  V_\beta \times R(\Gamma, \mathbf C^*) \hookrightarrow
 R(\Gamma,\mathrm{SL}_n(\mathbf C))
$$

\begin{lemma}
\label{lemma:T}
There exists a unique component $T$ of $R(\Gamma,\mathrm{SL}_n(\mathbf C))$ that contains 
$$
 V_\alpha \times  V_\beta \times R(\Gamma, \mathbf C^*).
$$
Moreover, we have $\dim T= n^2+n-3$.
\end{lemma}
\begin{proof}
By the hypothesis of Theorem~\ref{thm:str} we have
$\Delta_0^{\alpha\otimes\beta^*}(\lambda^n)\neq0$ and $\lambda^n$ is a simple root of 
$\Delta_1^{\alpha\otimes\beta^*}(t)$. Hence for all $\lambda'\neq\lambda$ which are
sufficiently close to $\lambda$ we have
$\Delta_q^{\alpha\otimes\beta^*}((\lambda')^n)\neq0$ for $q=0,1$.
Hence, by the argument in the proof of Proposition~\ref{prop:dim-H-M-lambda} we obtain
$H^q(\Gamma;\MM^\pm_{(\lambda')^{\pm n}})=0$ for $q=0,1$.

Now consider the representation 
\[
\rho_{\lambda'}= 
((\lambda')^{b\varphi}\otimes \alpha)\oplus ((\lambda')^{-a\varphi}\otimes \beta)\in V_\alpha \times  V_\beta \times R(\Gamma, \mathbf C^*)
\]
and the corresponding decomposition of $\mathfrak{sl}_n(\mathbf C)_{\Ad \rho_{\lambda'}}$   as   $\Gamma$-module:
$$
\mathfrak{sl}_n(\mathbf C)_{\Ad \rho_{\lambda'}}= 
\mathfrak{sl}_a(\mathbf C)_{\Ad \alpha}\oplus
\mathfrak{sl}_b(\mathbf C)_{\Ad \beta} \oplus \mathbf{C}\oplus
\MM^+_{(\lambda')^n}\oplus\MM^-_{(\lambda')^{-n}}.
$$
Hence 
\begin{align*}
\dim Z^1(\Gamma ; \mathfrak{sl}_n(\mathbf C)_{\Ad \rho_{\lambda'}}) &=
\dim H^1(\Gamma ; \mathfrak{sl}_n(\mathbf C)_{\Ad \rho_{\lambda'}}) +
\dim B^1(\Gamma ; \mathfrak{sl}_n(\mathbf C)_{\Ad \rho_{\lambda'}}) \\
&=
a-1+b-1+1 + n^2 -1 -1 = n^2+n-3.
\end{align*}
On the other hand, for the  $\mathrm{SL}_n(\mathbf C)$-orbit of 
$ V_\alpha \times  V_\beta \times R(\Gamma, \mathbf C^*)$
we have:
\[ \mathrm{SL}_n(\mathbf C)\cdot ( V_\alpha \times  V_\beta \times R(\Gamma, \mathbf C^*) )=
\mathrm{SL}_n(\mathbf C)/P^+ \cdot ( U^+ \cdot V_\alpha \times  V_\beta \times R(\Gamma, \mathbf C^*))
\]
where $U^+= \big\{ \big(
\begin{smallmatrix} \mathrm{id}_a &X \\ 0 & \mathrm{id_b} \end{smallmatrix}
\big)\;\big|\; X\in \MM^+_{(\lambda')^{n}}\big\}$. Now the action of $U^+$ on
$V_\alpha \times  V_\beta \times R(\Gamma, \mathbf C^*)$ is generically free since
$H^0(\Gamma;\MM^+_{(\lambda')^{n}})=0$ and hence
\begin{align*}
\dim \mathrm{SL}_n(\mathbf C)\cdot ( V_\alpha \times  V_\beta \times R(\Gamma, \mathbf C^*) )
&\geq ab+ ab +a^2+a -2 +b^2+b-2+1\\ &= n^2+n-3 \,.
\end{align*}
Therefore, $\rho_{\lambda'}$ is a smooth point of $R(\Gamma , \mathrm{SL}_n(\mathbf C))$ which is contained in a unique $n^2+n-3$-dimensional component 
$T$. Note that $T$ is the Zariski closure of the orbit 
$\mathrm{SL}_n(\mathbf C)\cdot ( V_\alpha \times  V_\beta \times R(\Gamma, \mathbf C^*)) $.
\end{proof}

Let $Y$ and $Z$ denote the components of the character variety that contain the
characters of $S$ and $T$ respectively. We have
$\dim Y=\dim S-\dim \mathrm{SL}_n(\mathbf C)= n-1$. In addition 
$\dim Z\geq a-1+b-1+1= n-1$ 
since $T$ contains  $V_\alpha \times  V_\beta \times R(\Gamma, \mathbf C^*) $. 
Notice that the generic dimension of the orbit of 
$(\alpha', \beta', \lambda')\in V_\alpha \times  V_\beta \times R(\Gamma, \mathbf C^*)$
is $n^2-2$. Hence, $\dim Z\leq \dim T - (n^2-2) = n-1$.
Hence $\dim Z=n-1$ and  $\dim T= n^2+n-3$.

Let $Z_\alpha\subset X(\Gamma,\mathrm{SL}_a(\mathbf C))$ and $Z_\beta\subset X(\Gamma,\mathrm{SL}_b(\mathbf C))$ denote the irreducible components that contain the 
respective projections of $V_\alpha$ and $V_\beta$. We have a commutative diagram
\[
\begin{CD}
V_\alpha \times  V_\beta \times R(\Gamma, \mathbf C^*) @>>>  
T\subset R(\Gamma,\mathrm{SL}_n(\mathbf C)) \\
@VVV @VVV\\
Z_\alpha\times Z_\beta\times \mathbf{C}^* @>>> Z\subset  X(\Gamma,\mathrm{SL}_n(\mathbf C))
\end{CD}
\]
The top row is injective but not the bottom one, as conjugation can realize permutations of rows and columns. In general those permutations are difficult to describe, 
but if we restrict to \new{irreducible} characters, this is simpler.

\begin{lemma}
\label{lemma:Z}
There exists a Zariski dense subset  $ \mathring{ Z}\subset Z$ such that:
\begin{itemize}
 \item[--] If  $Z_\alpha\neq Z_\beta$ (in particular if $a\neq b$), then  $ \mathring{ Z}\cong Z_\alpha^{irr}\times Z_\beta^{irr}\times \mathbf C^*$.
\item[--] If $Z_\alpha=Z_\beta$, then  $ \mathring{ Z}\cong Z_\alpha^{irr}\times Z_\beta^{irr}\times \mathbf C^*/\!\sim $,
where the relation is defined by 
$(\chi_a,\chi_b,\lambda)\sim (\chi_b,\chi_a,\lambda^{-1})$, for $(\chi_a,\chi_b,\lambda)\in Z_\alpha^{irr}\times Z_\beta^{irr}\times \mathbf C^*$.
\end{itemize}
\end{lemma}

Here $Z_\alpha^{irr}$ denotes the set of irreducible characters in $Z_\alpha$. We use similar notation for other components of 
charaters and representations.

\begin{proof}
Recall from the proof of Proposition~\ref{lemma:T} that $T$ is the Zariski closure of the orbit 
$\mathrm{SL}_n(\mathbf C)\cdot ( V_\alpha \times  V_\beta \times R(\Gamma, \mathbf C^*)) $.
As $V_\alpha^{irr}$ and $V_\beta^{irr}$ are dense in $V_\alpha$ and $V_\beta$, 
$\mathrm{SL}_n(\mathbf C)\cdot ( V_\alpha^{irr} \times  V_\beta^{irr} \times R(\Gamma, \mathbf C^*)) $
is dense in $T$. Its projection $ \mathring{ Z}$ to $Z$ is the image of 
$Z_\alpha^{irr}\times Z_\beta^{irr}\times \mathbf C^*$,
which is Zariski dense. 
To determine this image, we use that each point in $X(\Gamma,\mathrm{SL}_n(\mathbf C))$ is the character of a semi-simple
representation, unique up to conjugation \cite{LM85}. This uniqueness implies that for $Z_\alpha\neq Z_\beta$ this is an injective map,
and for $Z_{\alpha}=Z_{\beta}$ we quotient by the permutation of components,  with the corresponding
transformation for $\lambda$.
\end{proof}

 \begin{remark}
When $a=b=1$,  then $Z_\alpha=Z_\beta$ consists of a single point and $Z$ is the quotient of $\mathbf C^*$
by the involution $\lambda\mapsto 1/\lambda$. Hence $Z\cong\mathbf C$ and it is the variety of abelian
characters in $\SL(\mathbf C)$. The ring of functions
invariant by this involution  is generated by $\lambda+1/\lambda$, i.e. the trace of a diagonal 
matrix with eigenvalues $\lambda$ and $1/\lambda$ (corresponding to the  character evaluated
at a meridian).
   \end{remark}

%

We aim to show that $S$ and $T$ are the only components that contain
$\rho_\lambda$. For this purpose we consider the quadratic cone
$Q(\rho_\lambda)$ which is 
defined by the vanishing of an obstruction to integrability of $1$-cocycles. Let
$$
[.\smallcup.]\co 
H^1(\Gamma ; \mathfrak{sl}_n(\mathbf C)_{\Ad\rho_\lambda})\otimes H^1(\Gamma ;
\mathfrak{sl}_n(\mathbf C)_{\Ad\rho_\lambda})
\to H^2(\Gamma ; \mathfrak{sl}_n(\mathbf C)_{\Ad\rho_\lambda})
$$
denote the \emph{cup-bracket} which is the combination of the cup-product with
the Lie bracket
$\mathfrak{sl}_n(\mathbf C)\otimes \mathfrak{sl}_n(\mathbf
C)\xrightarrow{[\,.\,,\,.\,]} \mathfrak{sl}_n(\mathbf C)$.
The quadratic cone $Q(\rho_\lambda)\subset Z^1(\Gamma ;  \mathfrak{sl}_n(\mathbf
C)_{\Ad\rho_\lambda})$ is defined by
\[
Q(\rho_\lambda)=\{ \vartheta\in Z^1(\Gamma ;  \mathfrak{sl}_n(\mathbf C)_{\Ad\rho_\lambda}) \mid
[\vartheta\smallcup\vartheta]\sim0\}.
\]
Goldman showed \cite{Goldman1984} that if $\vartheta\in Z^1(\Gamma ; 
\mathfrak{sl}_n(\mathbf C)_{\Ad\rho_\lambda})$ is integrable then the cup-bracket
$[\vartheta\smallcup\vartheta]$ is a coboundary.
In what follows we will compute the projections of this obstruction,  for
 the projections
$$
\operatorname{pr}_\pm\co H^2(\Gamma ;  \mathfrak{sl}_n(\mathbf C)_{\Ad\rho_\lambda})\to H^2(\Gamma
; \MM^\pm_{\lambda^{\pm n}}).
$$
Here we use the decomposition of $\Gamma$-modules:
$$
\mathfrak{sl}_n(\mathbf C)_{\Ad\rho_\lambda}=\mathcal{D} \oplus
\MM^+_{\lambda^n}\oplus\MM^-_{\lambda^{-n}} = \mathfrak{sl}_a(\mathbf C)\oplus
\mathfrak{sl}_b(\mathbf C) \oplus \mathbf{C}\oplus
\MM^+_{\lambda^n}\oplus\MM^-_{\lambda^{-n}}.
$$
Recall that $\Gamma$ acts of $\mathfrak{sl}_n(\mathbf C)$,
$\mathfrak{sl}_a(\mathbf C)$ and 
$\mathfrak{sl}_b(\mathbf C)$ via the adjoint representation
$\mathrm{Ad}\,\rho_\lambda$,
$\mathrm{Ad}\,\alpha$ and $\mathrm{Ad}\,\beta$ respectively.
For the rest of this section we will understand these modules with this action.
Recall also that, by the hypothesis of Theorem~\ref{thm:str}, 
$\Delta_0^+(\lambda^n)\neq 0$ and $\lambda ^n$ is a simple root of $\Delta_1^+$. 
By Proposition~\ref{prop:dim-H-M-lambda} we have $\dim H^1(\Gamma ;
\MM^\pm_{\lambda^{\pm n}}) =1$
and we fix $d_\pm\in Z^1(\Gamma ; \MM^\pm_{\lambda^{\pm n}})$ which represent 
non-trivial cohomology classes.

Every element in $H^1(\Gamma ;  \mathfrak{sl}_n(\mathbf{C}))$ is represented by a cocycle
\begin{equation}
 \label{eqn:cocycle}
 \vartheta=
\begin{pmatrix}
 z_a & u_+ d_+\\
 u_- d_- & z_b
\end{pmatrix}
+z \varphi 
\begin{pmatrix}
    -b \operatorname{Id}_a & 0 \\
    0 & a \operatorname{Id}_b
\end{pmatrix},
\end{equation}
where $z_a\in Z^1(\Gamma ; \mathfrak{sl}_a(\mathbf{C}))$, $z_b\in Z^1(\Gamma ; \mathfrak{sl}_b(\mathbf{C}))$ and $u_\pm,  z\in\mathbf C$.

\begin{lemma}
 \label{lem:projections}
 For $\vartheta\in Z^1(\Gamma ; \mathfrak{sl}_n(\mathbf{C}))$ as in \eqref{eqn:cocycle}:
 \begin{align*}
  \operatorname{pr}_+[\vartheta\smallcup\vartheta] & \sim 2 u_+ (z_a\smallcup d_+ + d_+\smallcup z_b + z\, n\, d_+\smallcup \varphi), \\
  \operatorname{pr}_-[\vartheta\smallcup\vartheta] & \sim 2 u_- ( d_-\smallcup z_a + z_b\smallcup d_- - z\, n \, d_-\smallcup \varphi),
 \end{align*}
where $\sim$ denotes being  cohomologous.
\end{lemma}
\begin{proof}
The lemma follows from a direct calculation of 
\[[\vartheta\smallcup\vartheta](\gamma_1,\gamma_2) = [\vartheta(\gamma_1), \gamma_1\vartheta(\gamma_2)]\]
and Remark~\ref{rem:notwist}.
\end{proof}

In order to understand the cup products appearing in  Lemma~\ref{lem:projections} we  introduce the complex number 
$l_{\pm }(z_a,z_b)\in\mathbf C$. Consider
 a one-parameter analytical deformation $s\mapsto\alpha_s\oplus\beta_s$ of $\alpha\oplus\beta$ in  $V_\alpha \times  V_\beta $ tangent to $(z_a,z_b)$. 
 Notice that the coefficients of the twisted Alexander polynomial  $\Delta_1^{\alpha_s\otimes\beta^*_s}$ depend analytically on $s$.
 By the implicit function theorem and since $\lambda^n$ is a simple root of
 $\Delta_1^{\alpha\otimes\beta^*}$, there is an analytical path $s\mapsto r_s^+$ of roots of   $\Delta_1^{\alpha_s\otimes\beta^*_s}$ with $r_0^+=\lambda^n$.
Similarly there is a path $s\mapsto r_s^-$ of roots of $\Delta_1^{\beta_s\otimes\alpha^*_s}$ with $r_0^-=\lambda^{-n}$.
 We define
 $$
 l_\pm(z_a,z_b)=\left.\frac{d \phantom{s}}{ds}\right\vert_{s=0}\log r_s^\pm.
 $$
\begin{lemma} \label{lem:cup-and-log}
The following relations hold in $ Z^1(\Gamma ; \MM^\pm_{\lambda^{\pm n}})$:
 \label{lem:cupsandlogarithms}
 \begin{eqnarray*}
 z_a\smallcup d_+ + d_+\smallcup z_b & \sim &- l_+(z_a,z_b)\,  d_+\smallcup \varphi, \\
 d_-\smallcup z_a + z_b\smallcup d_-  & \sim &- l_-(z_a,z_b)\,   d_- \smallcup \varphi.
 \end{eqnarray*}
\end{lemma}

\begin{proof}
We know that $ z_a\smallcup d_+ + d_+\smallcup z_b$ is cohomologous to  $x\, d_+\smallcup\varphi$ for some $x\in\mathbf C$, as $H^2(\Gamma ; \mathcal{M}^+_\lambda)\cong \mathbf C$ 
and
$d_+\smallcup \varphi\not\sim 0$ (see Proposition~\ref{prop:dim-H-M-lambda} and Corollary~\ref{cor:cupnonzero}). Hence 
 by Lemma~\ref{lemma:connecting} the cocycle
\begin{equation*}
\zeta=
\begin{pmatrix}
 z_a & 0\\
 0 & z_b
\end{pmatrix}
+\frac xn \varphi 
\begin{pmatrix}
    - b \operatorname{Id}_a & 0 \\
    0 & a \operatorname{Id}_b
\end{pmatrix} \in Z^1(\Gamma ; \mathcal{D})
\end{equation*}
satisfies $\bock(\zeta)\sim 0$ where $\bock\co H^1(\Gamma ; \mathcal{D})\xrightarrow{\bock}H^2(\Gamma ; \MM^+_{\lambda^n})$ is the Bockstein operator of the exact cohomology sequence associated to \eqref{eqn:exactp}.

Furthermore, by Lemma~\ref{cor:h1p+} $\zeta$ is cohomologous to the restriction of a cocycle $\zeta^+\in Z^1(\Gamma ; \mathfrak{p}^+)$.
As  $\rho^+$ is a smooth point of $R(\Gamma, P^+)$ (Lemma~\ref{lemma:p+smooth}), we may consider a one parameter path $s\mapsto \rho_s$ in $R(\Gamma,P^+)$ tangent to $\zeta^+$ at $\rho^+$, that we write as:
$$
\rho_s =\begin{pmatrix}
                1 & d_s \\
                0 & 1
               \end{pmatrix}
\begin{pmatrix}
  \alpha_s \lambda_s^{b \varphi } & 0 \\
  0 & \beta_s \lambda_s^{-a \varphi }
\end{pmatrix}.
$$
In particular, by the definition of $\zeta$ we have that $s\mapsto\alpha_s$ is a deformation of $\alpha$ tangent to $z_a$,
 $s\mapsto\beta_s$ is a deformation of $\beta$ tangent to $z_b$, and
$\lambda_s=\lambda (1-\frac{x}{n} s + o(s^2) )$. By semi-continuity $d_s$ is a cocycle not cohomologous to zero because $d_0=d_+$, 
hence by Lemma~\ref{lem:delta10} we obtain
$\Delta_1^{\alpha_s\otimes \beta^*_s}(\lambda_s^n)=0$. Therefore, as 
$$
-\frac xn= \frac{\lambda_0'}{\lambda_0} =\left.\frac{d\phantom{s}}{ds}\right\vert_{s=0}\log \lambda_s,
$$
$x$ equals minus the derivative of the logarithm of the root of $\Delta_1^{\alpha_s\otimes \beta^*_s}$.
\end{proof}

Lemmas~\ref{lem:projections} and \ref{lem:cup-and-log} give:
\begin{cor}
\label{cor:pr_pm}
 For $\vartheta\in Z^1(\Gamma ; \mathfrak{sl}_n(\mathbf{C}))$ as in \eqref{eqn:cocycle}:
 $$
   \operatorname{pr}_\pm[\vartheta\smallcup\vartheta]  \sim 2 u_\pm ( -l_\pm(z_a,z_b) \pm  z\, n )\,  d_\pm\smallcup \varphi.
 $$
\end{cor}

Since $\Delta_1^{\alpha_s\otimes \beta^*_s}(t)=\Delta_1^{\beta_s\otimes \alpha^*_s}(1/t)$ by Corollary~\ref{cor:Alex-sym},
$$
l_+(z_a,z_b)= - l_-(z_a,z_b).
$$
Hence the vanishing of the obstructions to integrability of Corollary~\ref{cor:pr_pm} is equivalent to 
\begin{equation}
 \label{eq:obstruction}
 u_{+} (-l_+(z_a,z_b)+z\, n)=0 \quad \text{ and } \quad  u_{-} (-l_+(z_a,z_b)+z\, n)=0.
\end{equation}
Since $z$ can be interpreted as the derivative of the logarithm of $\lambda$, we view  $-l_+(z_a,z_b)+z\, n$
as the derivative of the difference between the logarithm of the root of the Alexander polynomial and  
the logarithm of $\lambda^n$.

Recall that by \eqref{eqn:cocycle} every cocycle  
$\vartheta\in Z^1(\Gamma ;  \mathfrak{sl}_n(\mathbf{C}))$ is of the form
\begin{equation}
\label{eqn:vartheta} 
 \vartheta=
\begin{pmatrix}
 z_a & u_+ d_+ + b_+\\
 u_- d_- + b_- & z_b
\end{pmatrix}
+z \varphi 
\begin{pmatrix}
    -b \operatorname{Id}_a & 0 \\
    0 & a \operatorname{Id}_b
\end{pmatrix},
\end{equation}
where $z_a\in Z^1(\Gamma ; \mathfrak{sl}_a(\mathbf{C}))$ and
$z_b\in Z^1(\Gamma ; \mathfrak{sl}_b(\mathbf{C}))$ are cocycles, 
$u_\pm,  z\in\mathbf C$, and
$b_\pm\in B^1(\Gamma; \MM^\pm_{\lambda^{\pm n}})$ are coboundaries.
  Notice 
that this formula differs from \eqref{eqn:cocycle} because here the coboundaries are also considered.

\begin{prop}
 \label{prop:tangentST} The Zariski tangent spaces at $\rho_\lambda$ are:
  \begin{eqnarray*}
   T_{\rho_\lambda} S & = & \{\vartheta\in Z^1(\Gamma ;  \mathfrak{sl}_n(\mathbf C))\mid -l_+(z_a,z_b)+z\, n=0\}, \\
   T_{\rho_\lambda} T & = & \{\vartheta\in Z^1(\Gamma ;  \mathfrak{sl}_n(\mathbf C))\mid u_+=u_-=0 \},
  \end{eqnarray*}
using the notation of \eqref{eqn:vartheta} for a cocycle 
$\vartheta\in Z^1(\Gamma ;  \mathfrak{sl}_n(\mathbf C))$.
  In particular $S$ and $T$ are smooth and  transverse at $\rho_{\lambda}$.
\end{prop}

\begin{proof}
 First at all, notice that $u_+$ is not identically zero on $ T_{\rho_\lambda} S$, by considering the tangent vector to the 
 one parameter path 
 $$ s\mapsto
 \begin{pmatrix}
  1 & s\, d_+ \\
  0 & 1
 \end{pmatrix}
\rho_{\lambda}\,.
 $$
 Then Equation~\eqref{eq:obstruction} implies $-l_+(z_a,z_b)+z\, n=0$ on
$T_{\rho_\lambda}S$.
 Furthermore, we know that $\dim S= n^2+n-2$ and, by Equation~\eqref{eq:dimZ1}
and Proposition~\ref{prop:dim-H-M-lambda}, the dimension of
$Z^1(\Gamma;\mathfrak{sl}_n(\mathbf C)_{\Ad\rho_\lambda})$ is $n^2+n-1$. 
 This shows the equality for $T_{\rho_\lambda}S$ and proves that $\rho_\lambda$
is a smooth point of $S$.
 
 We follow the same lines to prove the equality for $T_{\rho_\lambda}T$. Notice
that
$-l_+(z_a,z_b)+z\, n$ is not identically zero on $T_{\rho_\lambda}T$, by
considering deformations of 
$\lambda$ that keep $\alpha$ and $\beta$ constant.
 Hence  
 $u_+=u_-=0$ on $T_{\rho_\lambda}T$. Moreover, $\dim T=n^2+n-3$.
\end{proof}

We next compute the tangent space to character varieties at $\chi_\lambda$. 
Since the representation $\rho_\lambda$ is completely reducible, its orbit by
conjugation is closed, hence we can apply Luna's slice theorem 
as in   \cite{BenAbdelghani2002} or \cite[Section~9]{HP05}. As a consequence of
the slice theorem, since the centralizer of $\rho_\lambda$ is $\mathbf C^*$:
$$
T_{\chi_\lambda} X(\Gamma, \operatorname{SL}_n(\mathbf C))\cong H^1(\Gamma ; 
\mathfrak{sl}_n(\mathbf C)) \sslash\mathbf C^*.
$$
The action of $\mathbf C^*$ can be seen on the coordinates $u_\pm$: an element
$\varsigma\in \mathbf C^*$ maps $u_\pm$ to
$\varsigma^{\pm n}u_\pm$. Hence  we define a new coordinate 
$$
u=u_+ u_-
$$ 
and the obstructions \eqref{eq:obstruction} become
\begin{equation}
 \label{eq:obstructioncohom}
 u\, (-l_+(z_a,z_b)+z\, n)=0.
\end{equation}

Notice that even if $z_a$ and $z_b$ are cocycles,
the logarithmic derivative $-l_+(z_a,z_b)$ only depends on the cohomology class   of
$(z_a,z_b)$ in $H^1(\Gamma;  \mathfrak{sl}_a(\mathbf C)\oplus \mathfrak{sl}_b(\mathbf C))$. In addition,
$z$ is the scalar that describes a cohomology class  
$z\,\varphi\in H^1(\Gamma;\mathbf C) =Z^1(\Gamma;\mathbf C)\cong\mathbf C$.
Similarly for $u_{\pm}\in\mathbf C$ and the cohomology class $u_\pm [d_{\pm}]\in H^1(\Gamma;\MM^\pm_{\lambda^{\pm n}})\cong\mathbf C$. Thus we have the following:

\begin{remark} 
The obstruction in \eqref{eq:obstructioncohom} is well defined in $H^1(\Gamma ;  \mathfrak{sl}_n(\mathbf C)) \sslash\mathbf C^*$. 
\end{remark}

\begin{cor}
 \label{cor:TangenYZ} The Zariski tangent spaces to $Y$ and $Z$ are: 
  \begin{eqnarray*}
   T_{\chi_\lambda} Y & = & \{[\vartheta]\in   T_{\rho_\lambda} X(\Gamma, \operatorname{SL}_n(\mathbf C)) \mid -l_+(z_a,z_b)+z\, n=0\},\\
   T_{\chi_\lambda} Z & = & \{ [\vartheta]\in    T_{\rho_\lambda} X(\Gamma, \operatorname{SL}_n(\mathbf C))  \mid u=0 \}.
  \end{eqnarray*}
  In particular $Y$ and $Z$ are smooth and  transverse at $\chi_{\lambda}$.
\end{cor}

\begin{proof}
 The  proof is similar to Proposition~\ref{prop:tangentST}: we need to show that $u$ does not vanish on the Zariski tangent space to $Y$ and $-l_+(z_a,z_b)+z\, n$ does not vanish on
  the Zariski tangent space to
 $Z$. For the first assertion, we start with the cocycle \[\vartheta = \begin{pmatrix}
 0 & d_+ \\
 d_- & 0
\end{pmatrix} \in Z^1(\Gamma; \mathfrak{sl}_n(\mathbf C))\,.
\]
Following the notation of \eqref{eqn:vartheta}, since $z_a$, $z_b$ and $z$ vanish for this cocycle, 
Proposition~\ref{prop:tangentST} implies that
$\vartheta \in T_{\rho_\lambda} S$. In particular the projection of its cohomology class $\vartheta$ 
in 
$H^1(\Gamma ;  \mathfrak{sl}_n(\mathbf C)) \sslash \mathbf C^*$ is a vector tangent to $Y$ for which 
$u\neq 0$.
The proof that $-l_+(z_a,z_b)+z\, n$ is not identically zero on $ T_{\chi_\lambda} Z$ is the same as in Proposition~\ref{prop:tangentST}.
Then one concludes by using the dimension estimates.
\end{proof}

Notice that Corollary~\ref{cor:TangenYZ} and the computations of dimensions yield that 
$\chi_\lambda$ is a smooth point of both $Y$ and $Z$, and that $Y$ and $Z$ intersect transversally at $\chi_\lambda$.
In particular their intersection is a variety of dimension $n-2$. Since characters in this intersection must satisfy the condition on Alexander polynomials, we
have:

\begin{cor}
 \label{cor:interYZ} There is a neighborhood $\chi_\lambda\in U\subset  X(\Gamma, \operatorname{SL}_n(\mathbf C))$ such that
 $$
 (Y\pitchfork Z)\cap U=\Big\{ (\chi_{\alpha'},\chi_{\beta'},\lambda')\in Z 
 \cap U\:\big|\: \Delta^{\alpha'\otimes{(\beta')}^*}_1\big((\lambda')^n\big)=0\Big\}.
 $$
\end{cor}

 \section{An example}
 \label{section:example}

Let $K\subset S^3$ be the trefoil knot and $\Gamma=\pi_1(S^3\setminus\mathcal N(K))$. 
We use the presentation
$$
\Gamma\cong \langle x,y\mid x^2=y^3\rangle,
$$
in particular the center is the cyclic group generated by $z=x^2=y^3$. The abelianization map $\varphi\co\Gamma\to\mathbf Z$ satisfies $\varphi(x)=3$, $\varphi(y)=2$ and a meridian of the trefoil is given by $m=xy^{-1}$.

\begin{lemma}
\label{lemma:irrereps}
Every irreducible representation in $R (\Gamma,\mathrm{SL}_2(\mathbf C)) $ is conjugate to  $\alpha_s$, where 
\begin{equation}
\label{eq:rhos}
\alpha_s(x)=\begin{pmatrix}
           i & 0\\
           s & -i
          \end{pmatrix}
\quad\textrm{ and }\quad
\alpha_s(y)=\begin{pmatrix}
           \eta & \bar\eta-\eta\\
           0 & \bar \eta
          \end{pmatrix},
\end{equation}
for a unique $s\in\mathbf C$ and for $\eta\in\mathbf C$ a primitive sixth root of unity. Moreover, $\alpha_s$ is irreducible if and only if $s\neq 0,2i$. 
\end{lemma}
\begin{proof} 
Let $\alpha\co\Gamma\to\mathrm{SL}_2(\mathbf C)$ be an irreducible representation. Then by Schur's lemma
$\alpha(x^2)=\alpha(y^3)$ lies in the center $\{\pm \mathrm{id}_2\}$ of $\mathrm{SL}_2(\mathbf C)$.
If we had $\alpha(x)^2=\mathrm{id}_2$, then we would get $\alpha(x)=\pm\mathrm{id}_2$ and
$\alpha$ would be reducible, hence $\alpha(x)^2=\alpha(y^3)=  -\mathrm{id}_2$. Furthermore, as $\alpha(y)\neq -\mathrm{id}_2$,
the eigenvalues of $\alpha (y)$ are primitive sixth roots of unity.
The eigenspaces of $\alpha(x)$ 
and $\alpha(y)$ determine four points in $\mathbf P^1$. These four points are distinct since $\alpha$ is irreducible and by conjugation we can assume that $E_{\alpha(x)} (-i) = [0:1]$ is the point at infinity, $E_{\alpha(y)} (\eta) = [1 :0 ]$ and $E_{\alpha(y)} (\bar \eta) = [1 : 1 ]$. The last eigenspace
$E_{\alpha(x)} (-i) = [2i: s]= [1: -i s/2]$ determines the representation $\alpha$ up to conjugation.
Hence there exists $s\in\mathbf C$ such that $\alpha$ is conjugate to $\alpha_s$.
Moreover, the eigenspace
$E_{\alpha_s(x)} (-i)$ coincides with an eigenspace of $\alpha_s(y)$ if and only if $s\in\{0,2i\}$. 
\end{proof}

For any representation $\alpha\in R(\Gamma,\mathrm{SL}_2(\mathbf C))$ we consider the induced action on $\mathbf C^2$, as well as the action $\alpha\otimes t^\varphi$ on $\mathbf C^2[t^{\pm 1}]$.
We aim to compute the twisted Alexander polynomials $\Delta_0^\alpha$ and $\Delta_1^\alpha$, the  orders for the homology of $\alpha\otimes t^\varphi$. The quotient $\Delta_1^{\alpha_s}/ \Delta_0^{\alpha_s}$ has been calculated in a different way in \cite[Example~4.3]{Kitano-Morifuji2012}.

\begin{lemma}
 \label{lem:Delta1trefoil}
 For any irreducible $\alpha\in R(\Gamma,\mathrm{SL}_2(\mathbf C))$, we have
 $$
 \Delta_0^\alpha\doteq 1\ \text{ and }\  \Delta_1^\alpha\doteq t^2+1.
 $$
 \end{lemma}

\begin{proof}
First, we have $\Delta_0^\alpha\doteq 1$
since $\alpha$ is irreducible and $\dim \mathbf C^2>1$ (see \eqref{eq:irred-Delta0} in the proof of Lemma~\ref{prop:dualityDelta0}).

In order to calculate $\Delta_1^\alpha$ we will use the amalgamated product structure of $\Gamma$
$$
\Gamma\cong \langle x\rangle *_{\langle z\rangle} \langle y\rangle
$$
and the corresponding Mayer-Vietoris exact sequence in group homology \cite[VII.9]{Brown1982}.
We start computing some of the terms.
Since $\langle z\rangle\cong \mathbf Z$, the groups
 $H_q(\langle z\rangle, \mathbf C^2[t^{\pm 1}]_{\alpha\otimes t^\varphi})$ are the homology groups of the complex
 \[ 0\to \mathbf C^2[t^{\pm 1}]\xrightarrow{z-1} \mathbf C^2[t^{\pm 1}] \to 0\,.\]
 Hence a presentation matrix of 
 $H_0(\langle z\rangle, \mathbf C^2[t^{\pm 1}]_{\alpha\otimes t^\varphi})$ is
 $$
\begin{pmatrix}
           - t^6 -1& 0\\
           0 & - t^6-1
\end{pmatrix}.
 $$
Similarly, the respective presentation matrices for $H_0(\langle x\rangle, \mathbf C^2[t^{\pm 1}]_{\alpha\otimes t^\varphi})$ and $H_0(\langle y\rangle, \mathbf C^2[t^{\pm 1}]_{\alpha\otimes t^\varphi})$ 
are
 $$
\begin{pmatrix}
           i t^3 -1& 0\\
           0 & - i t^3-1
\end{pmatrix}
\quad\text{ and }\quad
\begin{pmatrix}
            e^{\frac{\pi i}{3}} t^2 -1& 0\\
           0 &  e^{\frac{-\pi i}{3}} t^2-1
\end{pmatrix}.
 $$
Since  $H_0(\langle x\rangle, \mathbf C^2[t^{\pm 1}]_{\alpha\otimes t^\varphi})$ and $H_0(\langle y\rangle, \mathbf C^2[t^{\pm 1}]_{\alpha\otimes t^\varphi})$ are torsion modules
it follows that $H_1(\langle x\rangle, \mathbf C^2[t^{\pm 1}]_{\alpha\otimes t^\varphi})\cong H_1(\langle y\rangle, \mathbf C^2[t^{\pm 1}]_{\alpha\otimes t^\varphi}) =0$.
Hence Mayer-Vietoris gives
a short exact sequence 
\begin{multline*}
0\to H_1(\Gamma ;  \mathbf C^2[t^{\pm 1}]_{\alpha\otimes t^\varphi})\to H_0(\langle z\rangle, \mathbf C^2[t^{\pm 1}]_{\alpha\otimes t^\varphi})
\to \\ H_0(\langle x\rangle, \mathbf C^2[t^{\pm 1}]_{\alpha\otimes t^\varphi}) \oplus H_0(\langle y\rangle, \mathbf C^2[t^{\pm 1}]_{\alpha\otimes t^\varphi})\to 0.
 \end{multline*}
Using this sequence and the presentation matrices:
 $$
 \Delta_1^\alpha= \frac{(t^6+1)^2}
{(t^3+i)(t^3-i)(t^2-e^{\frac{\pi i}{3}})(t^2-e^{\frac{-\pi i}{3}})}= t^2+1.
 $$
\end{proof}

It follows that Theorems~\ref{thm:suf} and~\ref{thm:str} apply for $\alpha$ irreducible and $\lambda\in\mathbf C$ satisfying $\lambda^6=-1$. Namely
Theorem~\ref{thm:suf} yields:

\begin{cor} When $\alpha\in R(\Gamma,\mathrm{SL}_2(\mathbf C))$ is irreducible and $\lambda^6=-1$,
 $$
(\lambda^{\varphi}\otimes \alpha)\oplus (\lambda^{-2\varphi}\otimes \mathbf{1})\co\Gamma\to \mathrm{SL}_3(\mathbf C)
$$
can be deformed to irreducible representations.
\end{cor}

To illustrate Theorem~\ref{thm:str}, we discuss next the variety of characters.

\paragraph{Varieties of characters}
The variety of characters $X(\Gamma,\mathrm{SL}_2(\mathbf C))$ has two
components, the abelian one and the one that contains irreducible
representations,
denoted by $X_0  (\Gamma,\mathrm{SL}_2(\mathbf C)) $. Let
$\chi_s\in 
X_0  (\Gamma,\mathrm{SL}_2(\mathbf C)) $ denote the character of $\alpha_s$
defined in \eqref{eq:rhos}.
The following is well known but we provide a proof for completeness and
because it is quite straightforward from Lemma~\ref{lemma:irrereps}.

\begin{lemma}
 \label{lem:isocharacters}
The map $s\mapsto \chi_s$ defines an isomorphism 
$
\mathbf C\cong X_0  (\Gamma,\mathrm{SL}_2(\mathbf C))
$.
\end{lemma}

\begin{proof} 
Using Lemma~\ref{lemma:irrereps}, the regular map 
$f\co\mathbf C\to X_0(\Gamma,\mathrm{SL}_2(\mathbf C))$ given by $f(s)= \chi_s$  restricts to a
bijection between $\{s\in\mathbf C\mid s\neq 0, 2i\}$ and the set of characters of irreducible representations 
$X^{irr}\subset X_0(\Gamma,\mathrm{SL}_2(\mathbf C))$.
A direct calculation gives for the meridian $m=xy^{-1}$ that $\chi_s(m)= i\bar\eta+ s(\bar\eta-\eta)$ is a linear function in $s$.
Hence there exists a regular map 
$g\co X_0(\Gamma,\mathrm{SL}_2(\mathbf C))\to \mathbf C$ such that
$g\circ f = \operatorname{id}_\mathbf{C}$. 
Since the image of $f$ contains $X^\mathit{irr}$, $f\circ g\circ f=f$
implies
$$
f\circ g\vert_ {X^\mathit{irr} }=  \operatorname{id}_{X^\mathit{irr} }.
$$
Both $f$ and $g$ are regular morphisms (defined on the whole variety,
not only on an open subset), hence
density yields:
$$
f\circ g=  \operatorname{id}_{X_0(\Gamma,\mathrm{SL}_2(\mathbf C))}
$$
establishing the isomorphism.
\end{proof}

For any $\lambda\in\mathbf C^*$ the map 
 $\alpha\mapsto (\lambda^{\varphi}\otimes
 \alpha)\oplus (\lambda^{-2\varphi}\otimes\mathbf{1})$
induces an embedding 
$$
i_{\lambda}:  X_0  (\Gamma,\mathrm{SL}_2(\mathbf C))\to
X(\Gamma,\mathrm{SL}_3(\mathbf C)).
$$
Let $X_{\lambda}=i_\lambda  (X_0  (\Gamma,\mathrm{SL}_2(\mathbf C)))$ denote its
image, that consists of characters of reducible representations.
We know that when $\lambda^6=-1$ $X_\lambda$ is contained in a two dimensional
component that contains irreducible characters.
Before describing the global structure of $X(\Gamma,\mathrm{SL}_3(\mathbf C))$,
we discuss the incidence between the  $X_\lambda$ when $\lambda^6=-1$.

Let
$
\tilde \sigma\co R(\Gamma,\mathrm{SL}_2(\mathbf C))\to \ 
R(\Gamma,\mathrm{SL}_2(\mathbf C))
$ be
the involution such that
$$
\tilde \sigma(\alpha)(x)= -\alpha(x)\quad \text{ and }\quad \tilde
\sigma(\alpha)(y)= \alpha(y),
$$
for every $\alpha \in R(\Gamma,\mathrm{SL}_2(\mathbf C))$, namely $\tilde
\sigma(\alpha)= (-1)^\varphi\otimes\alpha$. Denote  by $\sigma$ the induced
involution on $X_0(\Gamma,\mathrm{SL}_2(\mathbf C))$.
A straightforward computation gives
\[
 \tilde \sigma(\alpha)\mapsto (\lambda^{\varphi}\otimes
 \tilde\sigma(\alpha))\oplus (\lambda^{-2\varphi}\otimes\mathbf{1})= 
 ((-\lambda)^{\varphi}\otimes
 \alpha)\oplus (\lambda^{-2\varphi}\otimes\mathbf{1})
\]
and hence $i_\lambda\circ \sigma  =  i_{-\lambda}$.
It follows that $X_\lambda=X_{-\lambda}$. Notice also that 
$\tilde \sigma(\alpha_s)$ is conjugate to $\alpha_{2i-s}$. 

\begin{lemma}
\label{lem:threelines}
 For $\lambda\neq \pm \lambda'$ satisfying $\lambda^6=(\lambda')^6= -1$,
 $X_\lambda$ and $X_{\lambda'}$ intersect at a single point
$i_\lambda({\chi_s})$, with $s\in\{0,2i\}$.
In particular  $X_\lambda\cap X_{\lambda'}$ is the character of a diagonal
representation.
\end{lemma}

This gives a configuration of three lines $X_{e^{\pi i/6 }}$, $X_{i}$, $X_{e^{5
\pi i /6 }}$, that intersect pairwise at one point.
We shall prove that there is a  single component of
$X(\Gamma,\mathrm{SL}_3(\mathbf C))$ that contains irreducible representations,
and we shall describe how the three lines meet in this component.

\paragraph{Irreducible characters in $X(\Gamma,\mathrm{SL}_3(\mathbf C))$}
Let $\rho\in R(\Gamma,\mathrm{SL}_3(\mathbf C))$ be an irreducible representation. We denote
$ \rho(x)=A $ and $\rho(y)=B$. The matrix $A^2=B^3$ is a central element of $\mathrm{SL}_3(\mathbf C)$ because $\rho$ is irreducible.
The center of 
$\mathrm{SL}_3(\mathbf C)$ consists of three diagonal matrices
$\{\mathrm{id}_3,\omega\, \mathrm{id}_3, \omega^2\,\mathrm{id}_3\}$, where $\omega^2+\omega+1=0$. 

\begin{lemma}
 $A^2=B^3=\mathrm{id}_3$.
\end{lemma}

\begin{proof}
We need to exclude the cases $A^2=B^3=\omega \,  \mathrm{id}_3$ or equal to 
$\omega^2\,\mathrm{id}_3$.
Seeking  a contradiction, assume $A^2=B^3= \omega \,  \mathrm{id}_3$. The
equality $A^2=  \omega \,  \mathrm{id}_3$ implies that
one eigenvalue of $A$ has multiplicity at least two. Of course multiplicity
three is not compatible with irreducibility, thus
$A$ has a two-dimensional eigenspace. On the other hand, 
$B^3- \omega \,  \mathrm{id}_3=0$ combined with $\det(B)=1$
yields that the minimal polynomial of $B$ has also degree   two.
Hence $B$ has also a two dimensional eigenspace.
The intersection of the two dimensional eigenspaces of $A$ and $B$ 
is a proper invariant subspace, contradicting irreducibility.
The same argument applies to $A^2=B^3= \omega^2\,\mathrm{id}_3$.
\end{proof}

By the discussion in the proof of the previous lemma, the minimal polynomial of
$A$ is  $A^2-\mathrm{id}_3=0$ and the minimal polynomial of $B$
is $B^3-  \mathrm{id}_3=0$.
Therefore, the matrices $A$ and $B$ are conjugate
to
$$
A\sim \begin{pmatrix}
    1  & &\\
    & -1 & \\
    & & -1
\end{pmatrix}
\quad\textrm{ and }\quad
B\sim \begin{pmatrix}
    1  & &\\
    & \omega & \\
    & & \omega^2
\end{pmatrix},
$$
where $\omega^2+\omega+1=0$.
The corresponding eigenspaces are the plane  $E_A(-1)$ and the lines $E_A(1)$, 
$E_B(1)$, $E_B(\omega)$, and  $E_B(\omega^2)$.
The eigenspaces determine the representation, as they determine the matrices $A$
and $B$, that have fixed eigenvalues. Of course $E_A(1)\cap E_A(-1)=0$ 
and $E_B(1)$, $E_B(\omega)$, and  $E_B(\omega^2)$ are also in general position.
Since $\rho$ is irreducible, the five eigenspaces are in general position.
For instance  $E_A(1)\cap (E_B(1)\oplus E_B(\omega))=0$, because otherwise
$  E_B(1)\oplus E_B(\omega)=  E_A(1)\oplus (E_A(-1)\cap (E_B(1)\oplus
E_B(\omega)))$ would be a proper invariant subspace.

In order to parametrize the conjugacy classes of the irreducible representations, we fix some normalizations of those eigenspaces. The invariant lines correspond to fixed points in
the projective plane $\mathbf{P}^2$. The first normalization is that $E_A(-1)$ corresponds to the line at infinity, so that the 4 invariant lines are points in the
affine plane 
$\mathbf C^2$ in general position.  We further fix the three fixed points of $B$, corresponding to an affine frame. Then the fourth point (the line  $E_A(1)$) is a point in 
$\mathbf C^2$ that does not lie in the affine lines spanned  by any two of the fixed points of $B$.
This gives rise to the subvariety 
$\{\rho_{s,t}\in R(\Gamma,\mathrm{SL}_3(\mathbf C))\mid (s,t)\in\mathbf{C}^2\}$, where
the representation $\rho_{s,t}$ is given by  
\begin{equation}
\label{eq:rhost}
\rho_{s,t}(x)=\begin{pmatrix}
           1 & 0 &0\\
           s & -1 & 0\\
           t & 0 &-1
          \end{pmatrix}
\quad\textrm{ and }\quad
\rho_{s,t}(y)=\begin{pmatrix}
           1 & \omega -1 & \omega^2 -1\\
           0 & \omega & 0\\
           0 & 0 & \omega^2
          \end{pmatrix}\,.
\end{equation}
Here $\omega$ is a primitive 3rd root of unity, i.e.\ $\omega^2 + \omega +1 =0$.
The eigenspaces of $B$ determine points of $\mathbf{P}^2$:
\[ E_B(1) = [ 1:0:0],\ E_B(\omega) = [ 1:1:0], \text{ and }E_B(\omega^2) =
[1:0:1]. \]
The eigenspaces of $A$ determine a projective line (at infinity) and a point:
\[ E_A(-1)= \langle [0:1:0], [0:0:1]\rangle \text{ and }E_A(1) = [2:s:t].\]

\begin{lemma} \label{lem:(s,t)-irreductible} For $(s,t)\in \mathbf C^2$, the
representation $\rho_{s,t}$ is reducible
if and only if $s=0$, $t=0$, or $s+t =2$.
\end{lemma}

\begin{proof} 
The representation $\rho_{s,t}$ is constructed so that the points in $\mathbf{P}^2$ fixed by
$B=\rho_{s,t}(x)$ and 
the line  $E_A(-1)\subset \mathbf{P}^2 $ are fixed. So $\rho_{s,t}$ is reducible
iff the projective 
point $E_A(1)$ belongs to one of the lines spanned by two of the fixed points of
$B$. 
This condition is equivalent to one of the three equations $s=0$, $t=0$ or $s+t
=2$, one for each line. 
\end{proof}

It follows from the proof that, when $E_A(1)$ equals one of the fixed projective
points of $B$, then $A$ preserves also the two lines through that point that
are $B$-invariant. More precisely, we have:

\begin{remark}
\label{rem:intersectionoflines}
If two of the equations $\{s=0\}$, $\{t=0\}$ and $\{s+t=2\}$ hold true, then
$\rho_{s,t}$ preserves a complete flag in $\mathbf C^3$
and therefore it is conjugate to an upper triangular representation. Notice that
it has the same character as a diagonal representation.
\end{remark}

\begin{lemma}\label{lem:R3_1-irreducible}
Let $R^\mathit{irr}\subset R(\Gamma,\mathrm{SL}_3(\mathbf C))$ denote the subset
of irreducible representations. Then the Zariski closure 
$\overline{R^\mathit{irr}}\subset R(\Gamma,\mathrm{SL}_3(\mathbf C))$ is an
irreducible affine variety.
\end{lemma}

\begin{proof}
The variety $\mathbf{C}^2\times \mathrm{SL}_3(\mathbf C)$ is irreducible and the map
$\kappa\co\mathbf{C}^2\times \mathrm{SL}_3(\mathbf C)\to 
R(\Gamma,\mathrm{SL}_3(\mathbf C))$ given by $\kappa(s,t,D)= D\rho_{s,t} D^{-1}$
is a regular map. The image of $\kappa$ contains the irreducible representations
and every representation in the image of $\kappa$ is the limit of irreducible
representations. Hence
\[
R^\mathit{irr}\subset \kappa\big(\mathbf{C}^2\times \mathrm{SL}_3(\mathbf
C)\big)\subset
\overline{R^\mathit{irr}}
\]
and $\overline{\kappa\big(\mathbf{C}^2\times \mathrm{SL}_3(\mathbf C)\big)}=
\overline{R^\mathit{irr}}$ follows. Now the assertion of the lemma follows since
the closure of the image of an irreducible variety under a regular map is
irreducible.
\end{proof}

\begin{thm} 
\label{thm:trefoil}The GIT quotient
$X=\overline{R^\mathit{irr}}\sslash\mathrm{SL}(3,\mathbf{C})$ is isomorphic to
$\mathbf{C}^2$.
Moreover, the Zariski open subset $R^\mathit{irr}$ is
$\mathrm{SL}(3,\mathbf{C})$-invariant and its
GIT quotient is isomorphic to the complement of three affine lines in general
position in $\mathbf{C}^2$.
\end{thm}

\begin{proof}
By Lemma~\ref{lem:R3_1-irreducible}
the affine algebraic set
$\overline{R^\mathit{irr}}$ is irreducible. Since it is
$\mathrm{SL}(3,\mathbf{C})$-invariant, the GIT quotient  
$t\co\overline{R^\mathit{irr}}\to X$ exists and $X$ is also an irreducible
affine algebraic variety. Let $X^\mathit{irr}\subset X$ denote the projection of
$R^\mathit{irr}$,
which is Zariski open, in particular dense.

Consider the regular morphism $f\co \mathbf C^2\to X$ that maps  $(s,t)\in
\mathbf{C}^2$ to the character $\chi_{\rho_{s,t}}$.
By construction, the image of $f$ contains $X^\mathit{irr}$, because
$\rho_{s,t}$ realizes every irreducible representation up to conjugacy.

There is also a regular morphism
$ R(\Gamma,\mathrm{SL}_3(\mathbf C))\to \mathbf{C}^2$
given by
\[ \rho\mapsto \big(\mathrm{tr}\,\rho(m), \mathrm{tr}\,\rho(m^{-1})\big)\] 
where $m=xy^{-1}$ is a meridian of the trefoil knot, which induces (after
restriction)  a regular map 
$X \to \mathbf{C}^2$, by invariance. A direct computation gives:
\begin{equation}\label{eq:trace_meridian}
\begin{pmatrix}
\mathrm{tr}\,\rho_{s,t}(m)\\ \mathrm{tr}\,\rho_{s,t}(m^{-1}) 
\end{pmatrix}
=
\begin{pmatrix} 2\\2 \end{pmatrix}
 +
 \begin{pmatrix} 
 \omega^2-1 & \omega-1\\
\omega-1  & \omega^2-1
\end{pmatrix}
 \begin{pmatrix} s\\ t \end{pmatrix}\,.
\end{equation}%
Thus, after composing with a linear map, we have a regular morphism $g\co
X\to\mathbf{C}^2$ that satisfies
$$
g\circ f= \operatorname{id}_{\mathbf C^2}.
$$
Since the image of $f$ contains $X^\mathit{irr}$, $f\circ g\circ f=f$
implies
$$
f\circ g\vert_ {X^\mathit{irr} }=  \operatorname{id}_{X^\mathit{irr} }.
$$
Both $f$ and $g$ are regular morphisms (defined on the whole variety,
not only on an open subset), hence
density yields
$$
f\circ g=  \operatorname{id}_{X  },
$$
establishing the isomorphism.
\end{proof}
 
\begin{remark}
 It follows from Theorem~\ref{thm:trefoil} that the set of reducible characters in
$X\cong \mathbf C^2$ consists of three lines that intersect pairwise.
 Those are characters of representations $(\lambda^{-\varphi}\otimes\alpha)\oplus
(\lambda^{2\varphi}\otimes\mathbf 1)$, with $\alpha\in R(\Gamma,\operatorname{SL}_2(\mathbf C))$
irreducible except at the intersection points,
 that correspond to diagonal representations.
 
 Notice also that there is a  symmetry of  order three, as the center of
$\operatorname{SL}_3(\mathbf C)$ has order three. The symmetry group
is generated by
$$
\begin{array}{rcl}
 R(\Gamma,\mathrm{SL}_3(\mathbf{C})) & \to &  R(\Gamma,\mathrm{SL}_3(\mathbf{C})) \\
 \alpha & \mapsto &    \omega ^{\varphi}\otimes
	      \alpha
\end{array}
$$
where $\omega$ is a primitive  3rd root of unity.
This symmetry maps the character with coordinates $(s,t)$ to $(2-s-t, s)$, i.e.\ $\mathrm{tr}(\rho (m^{\pm 1}))$ to $\omega^{\pm 1}\mathrm{tr}(\rho ( m^{\pm 1}))$. Its
fixed point has coordinates $s=t=2/3$ (i.e.\ $ \mathrm{tr}(\rho ( m^{\pm 1}))=0$) 
and corresponds to the character of an irreducible metabelian representation.
This irreducible metabelian representation is obtained by composing the surjection
$\Gamma\twoheadrightarrow A_4$ with the $3$-dimensional irreducible representation of $A_4$
(see \cite{Serre1978}).
Note that irreducible, metabelian representations of knot groups into $\mathrm{SL}_n(\mathbf C)$ were studied by H.~Boden and S.~Friedl in a series of papers
\cite{Boden-Friedl2008, Boden-Friedl2011,Boden-Friedl2013,Boden-Friedl2014}.
\end{remark}

\begin{remark}
It is possible to combine any representation $\rho\co\Gamma\to\mathrm{SL}_2(\mathbf C)$ with the
irreducible $3$-dimensional rational representation of $r_3\co\mathrm{SL}_2(\mathbf{C})\to \mathrm{SL}_3(\mathbf{C})$ of
$\mathrm{SL}_2(\mathbf C)$ (for more details see \cite{Springer1977} and \cite{HM14}).
This induces a regular map 
\[ (r_3)_*\co  X_0  (\Gamma,\mathrm{SL}_2(\mathbf C))\to
X(\Gamma,\mathrm{SL}_3(\mathbf C))\,.\]
It follows from \cite[Prop.~3.1]{HM14} that the image of $(r_3)_*$ is contained in the component 
$X\subset X(\Gamma,\mathrm{SL}_3(\mathbf C))$. Notice that for every matrix $A\in\mathrm{SL}_2(\mathbf C)$ the equality
$\tr(r_3(A))=\tr(r_3(A)^{-1})$ holds. Then Equation~\eqref{eq:trace_meridian} implies  that the image of
$(r_3)_*$ is contained in the diagonal $\{s=t\}\subset \mathbf C^2\cong X$.
Moreover, the map $(r_3)_*$ factors through $X_0  (\Gamma,\mathrm{PSL}_2(\mathbf C))$ since 
$\Ker(r_3)=\{\pm\mathrm{id}\}$ is the center of $\mathrm{SL}_2(\mathbf C)$.
Hence $(r_3)_*$ is a 
two-fold branched covering onto its image.  The branching set is the character of the binary dihedral representation  
$d_6\co\Gamma\to D_6\subset\mathrm{SL}_2(\mathbf{C})$. Notice also that the restriction of 
$r_3$ onto $D_6$ becomes reducible, $r_3\circ d_6 \sim \rho_{1,1}$, since dihedral groups have only one and two-dimensional irreducible representations
(see \cite{Serre1978}).
\end{remark}

\begin{remark}
The same argument as in Theorem~\ref{thm:trefoil} applies to torus knots $T(p, 2)$, $p$ odd, to prove that the variety of irreducible
$\mathrm{SL}_3(\mathbf{C})$-characters 
consist of $(p-1)(p-2)/2$ disjoint components isomorphic to $\mathbf C^2$ and the components of reducible characters. 
\end{remark}

\bibliographystyle{alpha}

\begin{thebibliography}{BAHJ10}

\bibitem[AL02]{bAL2002}
Leila~Ben Abdelghani and Daniel Lines.
\newblock Involutions on knot groups and varieties of representations in a
  {L}ie group.
\newblock {\em J. Knot Theory Ramifications}, 11(1):81--104, 2002.

\bibitem[Art68]{Artin1968}
M.~Artin.
\newblock On the solutions of analytic equations.
\newblock {\em Invent. Math.}, 5:277--291, 1968.

\bibitem[BA00]{BenAbdelghani2000}
Leila Ben~Abdelghani.
\newblock Espace des repr\'esentations du groupe d'un n\oe ud classique dans un
  groupe de {L}ie.
\newblock {\em Ann. Inst. Fourier (Grenoble)}, 50(4):1297--1321, 2000.

\bibitem[BA02]{BenAbdelghani2002}
Leila Ben~Abdelghani.
\newblock Vari\'et\'e des caract\`eres et slice \'etale de l'espace des
  repr\'esentations d'un groupe.
\newblock {\em Ann. Fac. Sci. Toulouse Math. (6)}, 11(1):19--32, 2002.

\bibitem[BAHJ10]{bAHH2010}
Leila Ben~Abdelghani, Michael Heusener, and Hajer Jebali.
\newblock Deformations of metabelian representations of knot groups into {${\rm
  SL}(3,{\bf C})$}.
\newblock {\em J. Knot Theory Ramifications}, 19(3):385--404, 2010.

\bibitem[BF08]{Boden-Friedl2008}
Hans~U. Boden and Stefan Friedl.
\newblock Metabelian {${\rm SL}(n,\Bbb C)$} representations of knot groups.
\newblock {\em Pacific J. Math.}, 238(1):7--25, 2008.

\bibitem[BF11]{Boden-Friedl2011}
Hans~U. Boden and Stefan Friedl.
\newblock Metabelian {${\rm SL}(n,\Bbb C)$} representations of knot groups,
  {II}: fixed points.
\newblock {\em Pacific J. Math.}, 249(1):1--10, 2011.

\bibitem[BF13]{Boden-Friedl2013}
Hans~U. Boden and Stefan Friedl.
\newblock Metabelian {${\rm SL}(n,\Bbb C)$} representations of knot groups
  {III}: deformations.
\newblock {\em Q. J. Math.}, 2013.
\newblock Advance Access published October 9, 2013.

\bibitem[BF14]{Boden-Friedl2014}
Hans~U. Boden and Stefan Friedl.
\newblock Metabelian {${\rm SL}(n,\Bbb C)$} representations of knot groups
  {IV}: twisted {A}lexander polynomials.
\newblock {\em Math. Proc. Cambridge Philos. Soc.}, 156(1):81--97, 2014.

\bibitem[Bro94]{Brown1982}
Kenneth~S. Brown.
\newblock {\em Cohomology of groups}, volume~87 of {\em Graduate Texts in
  Mathematics}.
\newblock Springer-Verlag, New York, 1994.
\newblock Corrected reprint of the 1982 original.

\bibitem[Bur67]{Burde1967}
Gerhard Burde.
\newblock Darstellungen von {K}notengruppen.
\newblock {\em Math. Ann.}, 173:24--33, 1967.

\bibitem[dR67]{dR1967}
Georges de~Rham.
\newblock Introduction aux polyn\^omes d'un n\oe ud.
\newblock {\em Enseignement Math. (2)}, 13:187--194 (1968), 1967.

\bibitem[FK91]{FK1991}
Charles~D. Frohman and Eric~P. Klassen.
\newblock Deforming representations of knot groups in {${\rm SU}(2)$}.
\newblock {\em Comment. Math. Helv.}, 66(3):340--361, 1991.

\bibitem[FKK12]{FTT12}
Stefan Friedl, Taehee Kim, and Takahiro Kitayama.
\newblock Poincar\'e duality and degrees of twisted {A}lexander polynomials.
\newblock {\em Indiana Univ. Math. J.}, 61(1):147--192, 2012.

\bibitem[Fra37]{Franz37}
Wolfgang Franz.
\newblock Torsionsideale, torsionsklassen und torsion.
\newblock {\em Journal für die reine und angewandte Mathematik}, 176:113--125,
  1937.

\bibitem[Gol84]{Goldman1984}
William~M. Goldman.
\newblock The symplectic nature of fundamental groups of surfaces.
\newblock {\em Adv. in Math.}, 54(2):200--225, 1984.

\bibitem[Har77]{Hartshorne1977}
Robin Hartshorne.
\newblock {\em Algebraic geometry}.
\newblock Springer-Verlag, New York, 1977.
\newblock Graduate Texts in Mathematics, No. 52.

\bibitem[HK97]{HK1997}
Michael Heusener and Eric Klassen.
\newblock Deformations of dihedral representations.
\newblock {\em Proc. Amer. Math. Soc.}, 125(10):3039--3047, 1997.

\bibitem[HK98]{HK1998}
Michael Heusener and Jochen Kroll.
\newblock Deforming abelian {${\rm SU}(2)$}-representations of knot groups.
\newblock {\em Comment. Math. Helv.}, 73(3):480--498, 1998.

\bibitem[HM14]{HM14}
Michael Heusener and Ouardia Medjerab.
\newblock Deformations of reducible representations of knot groups into
  $\mathrm{SL}(n,\mathbf{C})$.
\newblock arXiv:1402.4294, 2014.

\bibitem[HP05]{HP05}
Michael Heusener and Joan Porti.
\newblock Deformations of reducible representations of 3-manifold groups into
  {${\rm PSL}_2(\Bbb C)$}.
\newblock {\em Algebr. Geom. Topol.}, 5:965--997, 2005.

\bibitem[HP11]{Heusener-Porti2011}
Michael Heusener and Joan Porti.
\newblock Infinitesimal projective rigidity under {D}ehn filling.
\newblock {\em Geom. Topol.}, 15(4):2017--2071, 2011.

\bibitem[HPSP01]{HPS2001}
Michael Heusener, Joan Porti, and Eva Su{\'a}rez~Peir{\'o}.
\newblock Deformations of reducible representations of 3-manifold groups into
  {${\rm SL}_2(\bold C)$}.
\newblock {\em J. Reine Angew. Math.}, 530:191--227, 2001.

\bibitem[HSW10]{HSW10}
Jonathan~A. Hillman, Daniel~S. Silver, and Susan~G. Williams.
\newblock On reciprocality of twisted {A}lexander invariants.
\newblock {\em Algebr. Geom. Topol.}, 10(2):1017--1026, 2010.

\bibitem[Hum75]{Humphreys1975}
James~E. Humphreys.
\newblock {\em Linear algebraic groups}.
\newblock Springer-Verlag, New York, 1975.
\newblock Graduate Texts in Mathematics, No. 21.

\bibitem[Hum95]{Humphreys1995}
James~E. Humphreys.
\newblock {\em Conjugacy classes in semisimple algebraic groups}, volume~43 of
  {\em Mathematical Surveys and Monographs}.
\newblock American Mathematical Society, Providence, RI, 1995.

\bibitem[JM87]{Johnson-Millson1987}
Dennis Johnson and John~J. Millson.
\newblock Deformation spaces associated to compact hyperbolic manifolds.
\newblock In {\em Discrete groups in geometry and analysis ({N}ew {H}aven,
  {C}onn., 1984)}, volume~67 of {\em Progr. Math.}, pages 48--106. Birkh\"auser
  Boston, Boston, MA, 1987.

\bibitem[Kit96]{Kitano96}
Teruaki Kitano.
\newblock Twisted {A}lexander polynomial and {R}eidemeister torsion.
\newblock {\em Pacific J. Math.}, 174(2):431--442, 1996.

\bibitem[KL99]{KirkLivingston99}
Paul Kirk and Charles Livingston.
\newblock Twisted {A}lexander invariants, {R}eidemeister torsion, and
  {C}asson-{G}ordon invariants.
\newblock {\em Topology}, 38(3):635--661, 1999.

\bibitem[KM12]{Kitano-Morifuji2012}
Teruaki Kitano and Takayuki Morifuji.
\newblock Twisted {A}lexander polynomials for irreducible {$SL(2,\Bbb
  C)$}-representations of torus knots.
\newblock {\em Ann. Sc. Norm. Super. Pisa Cl. Sci. (5)}, 11(2):395--406, 2012.

\bibitem[LM85]{LM85}
Alexander Lubotzky and Andy~R. Magid.
\newblock Varieties of representations of finitely generated groups.
\newblock {\em Mem. Amer. Math. Soc.}, 58(336):xi+117, 1985.

\bibitem[Mil62]{Milnor62}
John Milnor.
\newblock A duality theorem for {R}eidemeister torsion.
\newblock {\em Ann. of Math. (2)}, 76:137--147, 1962.

\bibitem[Nag62]{Nagata1961}
Masayoshi Nagata.
\newblock Complete reducibility of rational representations of a matric group.
\newblock {\em J. Math. Kyoto Univ.}, 1:87--99, 1961/1962.

\bibitem[New78]{Newstead1978}
P.~E. Newstead.
\newblock {\em Introduction to moduli problems and orbit spaces}, volume~51 of
  {\em Tata Institute of Fundamental Research Lectures on Mathematics and
  Physics}.
\newblock Tata Institute of Fundamental Research, Bombay; by the Narosa
  Publishing House, New Delhi, 1978.

\bibitem[Pop08]{Popov2008}
Vladimir~L. Popov.
\newblock Irregular and singular loci of commuting varieties.
\newblock {\em Transform. Groups}, 13(3-4):819--837, 2008.

\bibitem[Por97]{Porti}
Joan Porti.
\newblock Torsion de {R}eidemeister pour les vari\'et\'es hyperboliques.
\newblock {\em Mem. Amer. Math. Soc.}, 128(612):x+139, 1997.

\bibitem[{Ric}79]{Richardson1979}
R.W. {Richardson}.
\newblock {Commuting varieties of semisimple Lie algebras and algebraic
  groups.}
\newblock {\em {Compos. Math.}}, 38:311--327, 1979.

\bibitem[Ser78]{Serre1978}
Jean-Pierre Serre.
\newblock {\em Repr\'esentations lin\'eaires des groupes finis}.
\newblock Hermann, Paris, revised edition, 1978.

\bibitem[Sha77]{Shafarevich1977}
I.~R. Shafarevich.
\newblock {\em Basic algebraic geometry}.
\newblock Springer-Verlag, Berlin, study edition, 1977.
\newblock Translated from the Russian by K. A. Hirsch, Revised printing of
  Grundlehren der mathematischen Wissenschaften, Vol. 213, 1974.

\bibitem[Sha94]{AlgebraicGeometryIV}
I.~R. Shafarevich, editor.
\newblock {\em Algebraic geometry. {IV}}, volume~55 of {\em Encyclopaedia of
  Mathematical Sciences}.
\newblock Springer-Verlag, Berlin, 1994.
\newblock Linear algebraic groups. Invariant theory, A translation of {{\i}t
  Algebraic geometry. 4} (Russian), Akad. Nauk SSSR Vsesoyuz. Inst. Nauchn. i
  Tekhn. Inform., Moscow, 1989 [ MR1100483 (91k:14001)], Translation edited by
  A. N. Parshin and I. R. Shafarevich.

\bibitem[Sho91]{Shors1991}
Douglas~James Shors.
\newblock {\em Deforming reducible representations of knot groups in
  {SL}(2)({C})}.
\newblock ProQuest LLC, Ann Arbor, MI, 1991.
\newblock Thesis (Ph.D.)--University of California, Los Angeles.

\bibitem[Spr77]{Springer1977}
T.~A. Springer.
\newblock {\em Invariant theory}.
\newblock Lecture Notes in Mathematics, Vol. 585. Springer-Verlag, Berlin,
  1977.

\bibitem[Tur86]{Turaev86}
V.~G. Turaev.
\newblock Reidemeister torsion in knot theory.
\newblock {\em Uspekhi Mat. Nauk}, 41(1(247)):97--147, 240, 1986.

\bibitem[Wad94]{Wada94}
Masaaki Wada.
\newblock Twisted {A}lexander polynomial for finitely presentable groups.
\newblock {\em Topology}, 33(2):241--256, 1994.

\bibitem[Wei64]{Weil64}
Andr{\'e} Weil.
\newblock Remarks on the cohomology of groups.
\newblock {\em Ann. of Math. (2)}, 80:149--157, 1964.

\end{thebibliography}

\bigskip

\textsc{Universit\'e Clermont Auvergne, Université Blaise Pascal, 
Laboratoire de Math\'ematiques, BP 10448, F-63000 Clermont-Ferrand.\\
CNRS, UMR 6620, LM, F-63171 Aubi\`ere, FRANCE},
heusener@math.univ-bpclermont.fr.

\smallskip

\textsc{Departament de Matem\`atiques, Universitat Aut\`onoma de
Barcelona, 08193 Cerdanyola del Vall\`es, Spain},  porti@mat.uab.cat.

\end{document}